\documentclass{birkau}
  \usepackage[ruled,linesnumbered]{algorithm2e}
  \usepackage{hyperref}
  \usepackage{stmaryrd}
  \usepackage[english]{babel}
  \usepackage[utf8]{inputenc}
  \usepackage{amsfonts,amsmath,amssymb,enumerate,latexsym}

  \let\phi\varphi

  \let\subset\subseteq
  \def\compEXPTIME{{\textsf{EXPTIME}}}
  \def\compNP{{\textsf{NP}}}
  \def\compP{{\textsf{P}}}

  \def\algA{{\mathbf{A}}}

  \newcommand\vect[1]{\mathbf {#1}}
  \newcommand\textremovalmachine[1]{}

  \newcommand\en{\mathbb{N}}
  \newcommand\seq[1]{\vect{seq}(#1)}
  
  \DeclareMathOperator\Sg{Sg}
  
  \DeclareMathOperator\cont{cont}

  \theoremstyle{plain}
  \newtheorem{theorem}{Theorem}
  \newtheorem{lemma}[theorem]{Lemma}
  \newtheorem{claim}[theorem]{Claim}
  \newtheorem{corollary}[theorem]{Corollary}

  \newtheorem{observation}[theorem]{Observation}
  \newtheorem{proposition}[theorem]{Proposition}
  \theoremstyle{definition}
  \newtheorem{definition}[theorem]{Definition}

  \newtheorem{remark}[theorem]{Remark}

\begin{document}

  \author{Alexandr Kazda}
\address{Department of Algebra, Charles University\\
Sokolovsk\'a 83, 186 00 Praha 8, Czech Republic\\
ORCID 0000-0002-7338-037X\\
}
\email{alex.kazda@gmail.com}
\author{Dmitriy Zhuk}
\address{Department of Mathematics and Mechanics, Moscow State University\\
Moscow, Russia\\
}
\email{zhuk.dmitriy@gmail.com}
  \title{Existence of cube terms in finite algebras}
\thanks{The first author was supported by the PRIMUS/SCI/12 and
UNCE/SCI/22 projects of the Charles University and by the the Czech Science Foundation project GA \v CR 18-20123S.
The second author was supported by Russian Foundation for Basic Research  (grant 19-01-00200).
Our thanks go to Ralph Freese who has, with amazing speed, implemented Algorithm~\ref{algBlocker} in UAcalc.}


  \begin{abstract}
    We study the problem of whether a given finite algebra with finitely many basic operations
    contains a cube term; we give both structural and algorithmic results. We show that if such an algebra has a
    cube term then it has a cube term of dimension at most $N$, where the number $N$
    depends on the arities of basic operations of the algebra and the size of the
    basic set. For finite idempotent algebras we give a tight bound on $N$ that, in the special case of algebras with more than $\binom{|A|}2$  basic operations, improves
    an earlier result of K. Kearnes and A. Szendrei. On the algorithmic
    side, we show that deciding the existence of cube terms is in \compP{} for idempotent algebras and in \compEXPTIME{} in general. 
    
    Since an algebra contains a $k$-ary near unanimity operation if and only if it
    contains a $k$-dimensional cube term and generates a congruence distributive variety, our algorithm also lets us decide whether a given finite algebra has a near unanimity operation.  
  \end{abstract}
\maketitle

  \section{Introduction}
 Our goal in this paper is to explore the conditions for the existence of a
  cube term in finite a algebra (with finitely many basic
  operations).
  A finite algebra has few subalgebras of powers if and only if it has a cube term
(equivalently, an edge term or a parallelogram term) of some dimension.
  See~\cite{BIMMVW} and~\cite{kearnes-szendrei-parallelogram}
  for an introduction to these terms,
  and~\cite{idzak-markovic-mckenzie-valeriote-willard-tractability} for the
  application to the Constraint Satisfaction Problem.

  How to efficiently decide if a given finite algebra has a cube term? This
  question has practical significance: When coming up with hypotheses about few
  subpowers, one might want to quickly know if a candidate for a counterexample
  actually has few subpowers.  Deciding existence of various operations is also
  of interest in computational universal algebra  (e.g. in the software package
  UAcalc~\cite{UAcalc}).
  
  Conveniently, deciding whether an algebra $\algA$ has a cube term
  goes a long way towards telling us whether $\algA$ also has a near unanimity
  operation as $\algA$ has a near unanimity operation if and only if $\algA$ has a cube term and  $\algA$ generates a congruence distributive variety. 
  
  Since we present an \compEXPTIME{} algorithm for the former property and deciding the latter property is \compEXPTIME{}-complete for general algebras~\cite{freese-valeriote-maltsev-conditions}, we can place the problem of deciding if an algebra has a near unanimity term into the class \compEXPTIME. This is a significant improvement over the previous result~\cite{maroti-nu-is-decidable} that the problem is algorithmically decidable.

Besides the general case, we also look at deciding the existence of cube terms for idempotent algebras. Our
Algorithm~\ref{algBlocker} shows how to do this in polynomial time. This improves a previous algorithm given in~\cite[Corollary~3.6]{markovic-maroti-mckenzie-blockers}. Previously, a polynomial time algorithm was known for deciding cube terms in idempotent algebras when the support set $A$ was fixed (from the proof of~\cite[Corollary~3.6]{markovic-maroti-mckenzie-blockers} it in fact follows that the problem is fixed parameter tractable when parameterized by $|A|$). Our Algorithm~\ref{algBlocker} improves this result by making the running time depend polynomially on $|A|$. 

Finally, a problem related to deciding cube terms is to decide if a given algebra has a cube term of a given fixed dimension. Horowitz investigates this problem~\cite{horowitz-ijac} and, among other things, gives a polynomial time to decide whether an idempotent algebra has a cube term of a given fixed dimension. However, the local to global method from~\cite{horowitz-ijac} does not easily extend to the situation when we want a cube term of any dimension.

  Before we begin, we would like to point out that the part of our paper
  devoted to idempotent algebras (Section~\ref{secCTBidemp}) has a significant overlap with the
  results in~\cite{kearnes-szendrei-cube-term-blockers} by Keith Kearnes and
  Agnes Szendrei. To be specific, \cite[Theorem
  4.1]{kearnes-szendrei-cube-term-blockers} is a version of our
  Theorem~\ref{thmIdempotentCTB}. 
  Both our Theorem~\ref{thmIdempotentCTB} and \cite[Theorem 4.1]{kearnes-szendrei-cube-term-blockers} talk about idempotent algebras with finitely many basic operations. Unlike Theorem~4.1, our Theorem~\ref{thmIdempotentCTB} only applies to finite algebras.  However, on its home ground  Theorem~\ref{thmIdempotentCTB} sometimes offers a better upper bound on the dimension of a cube term; this happens for idempotent algebras with more than $\binom{|A|}{2}$ basic operations. This makes Theorem~\ref{thmIdempotentCTB} better suited for examining finite algebras. In fact, we use it to obtain a bound on the minimum dimension of a cube term that depends quadratically on $|A|$ and linearly on the maximum arity of a basic operation of $\algA$ (Corollary~\ref{cor:quadratic-linear}).
  
  Another topic our paper shares with~\cite{kearnes-szendrei-cube-term-blockers} is the construction of algebras
  proving that the bound on the dimension of a
  cube term from Theorem~\ref{thmIdempotentCTB} (or \cite[Theorem
  4.1]{kearnes-szendrei-cube-term-blockers}) is tight; this is Theorem~\ref{thmIdempotentTight} here and Example~4.4 in \cite{kearnes-szendrei-cube-term-blockers}. While the outcome is
  similar, our construction is novel in that it works for any (finite, greater
  than 2) size of the base set.
  
  When constructing the proof of Theorem~\ref{thmIdempotentCTB}, we had heard
  the statement of Theorem~4.1 in~\cite{kearnes-szendrei-cube-term-blockers},
  thus priority for the result belongs to Keith Kearnes and Agnes Szendrei. It
  also turns out that our methods for idempotent algebras resemble those
  of~\cite{kearnes-szendrei-cube-term-blockers} (in particular, our ``chipped
  cubes'' generalize the ``crosses'' of Keith Kearnes and Agnes Szendrei). 
  However, we produced the proof of Theorem~\ref{thmIdempotentCTB} on our own
  as we only knew the statement, not the proof of~\cite[Theorem
  4.1]{kearnes-szendrei-cube-term-blockers} when discovering our proof. 

  We include a full
  proof of Theorem~\ref{thmIdempotentCTB} in this paper because it illustrates
  the ideas we later develop for the non-idempotent case and also because
  Theorem~\ref{thmIdempotentCTB} naturally leads to
  Theorem~\ref{thmIdempotentTight} which shows that the bounds on cube term
  dimension of Theorem~\ref{thmIdempotentCTB} are tight for all applicable sizes of
  the base set of the algebra in question, improving the state of the art.

  \section{Preliminaries}
  We will be using basic notations such as operations, algebra of a given signature, clone of operations of an algebra, product of algebras, subuniverse, subuniverse generated by a set, variety and identity (equation) as is usual in universal algebra. See, e.g,~\cite{burris,uabook}.
  
  We will spend much time and effort designing and examining tuples.  A \emph{tuple} on
  $A$ of arity $n$ (or $n$-tuple on $A$) is a sequence of $n$ members of the set
  $A$. If we want to emphasize that an object is a tuple we print it in bold:
  $\vect a$.  The set of all $n$-tuples on $A$ will be denoted by $A^n$.  If
  $\vect a\in A^n$ and $i\in \en$, we denote the $i$-th entry of $\vect a$ by
  $a_i$ or sometimes by $(\vect a)_i$.  On the other hand, indices without
  parentheses and in bold shall refer to a particular member of a sequence of
  tuples, so e.g.  $\vect{a_3}$ is the third tuple from some sequence, not the
  third entry of $\vect a$. If confusion is unlikely, we will write the tuple
  $(a_1,a_2,\dots,a_n)$ in a more compact form as $a_1a_2\dots a_n$.

  When $\vect a\in A^n,\vect b\in A^k$ are tuples, we will denote by $\vect
  a\vect b$ their concatenation, i.e. the tuple $a_1a_2\dots a_nb_1 b_2\dots
  b_k\in A^{n+k}$. 
  
  If $i\leq j$ are positive integers and $\vect a$ is an $n$-tuple, then by
  $\vect a_{[i,j]}$ we mean the $(j-i+1)$-tuple $a_ia_{i+1}\dots a_{j}$.

  If $a\in A$ and $k\in\en$ then $a^{\bf k}$ is the $k$-tuple whose all entries are $a$ 
  (the boldface $k$ signals that we are turning $a$ into a tuple),
  i.e. $a^{\bf k}=aa\dots a$. If $\vect a=a_1^{\bf
  n_1}a_2^{\bf n_2}\dots a_k^{\bf n_k}$ for some $a_1,\dots,a_k\in A$, we will
  call the interval $B_i=\{j\colon n_1+\dots+n_{i-1} <j\leq n_1+\dots+n_i \}$
  the $i$-th \emph{block} of $\vect a$. The partition into blocks can be 
  ambiguous if e.g. $a_1=a_2$, but this will not be an issue as we will 
  usually fix the partition in advance. A careful reader might have noticed
  that we have overloaded e.g. $a_2$ to mean both the second element and the
  $(n_1+1)$-th element of $\vect a$. We did this to keep our notation short; 
  when $\vect a$ is broken into obvious blocks then $a_2$ always stands for the
  element forming up the second block of $\vect a$.

  Finally, if $A$ is a (finite) set then $\seq{A}$ is the $|A|$-tuple that lists all 
  elements of $A$ in a fixed order (for example, $\seq{\{3,1,4\}}=(1,3,4)$). 
  We will denote the set $\{1,2,\dots,n\}$ by $[n]$.

  The following way to combine tuples, introduced in~\cite{BIMMVW}, will be useful when talking about cube
  terms and blockers:
  \begin{definition}
    Let $n\in\en$, $\vect a,\vect b\in A^n$ be two tuples. We then define for
    each $I\subset [n]$ the $n$-tuple $\chi_I(\vect a,\vect b)$ by
  \[
    (\chi_I(\vect a,\vect b))_i=\begin{cases}
     b_i&\text{if}\; i\in I\\
     a_i&\text{if}\; i\not\in I.\\
   \end{cases}
  \]
  
  \end{definition}

  In particular, we will often consider the matrix $(\chi_I(\vect a,\vect
  b)\colon I\neq \emptyset)$. If $\vect a,\vect b\in A^n$, then $(\chi_I(\vect a,\vect
  b)\colon I\neq \emptyset)$ is an
  $n\times (2^n-1)$ matrix whose $i$-th column is $\chi_I(\vect a,\vect b)$
  where $I$ is the $i$-th nonempty subset of $[n]$ (ordered in some fixed way).

We will often apply a term operation to a matrix. Let $\algA$
  be an algebra and let $t$ be an $m$-ary operation of $\algA$. If $M$ is an
  $n\times m$ matrix, then $t(M)$ is the $n$-tuple that we obtain by applying
  $t$ on the rows of $M$.


For $E\subset A^n$, we will denote by $\Sg_{\algA^n}(E)$ the subuniverse of $\algA^n$
  generated by $E$. 
  If we want to emphasize the dimension and the algebra we are
  talking about, we will write $\Sg_{\algA^n}(E)$. To simplify
  notation, we will often omit curly brackets in the argument of $\Sg$, writing
  e.g. $\Sg(u,v)$ instead of $\Sg(\{u,v\})$.

  The algebra $\algA$ is \emph{idempotent} if the identity
  $t(a,a,\dots,a)=a$
  holds for each term operation $t$ of $\algA$ and all $a\in A$ (this is equivalent to demanding that the identity holds for 
  each of $\algA$'s basic operations).


  Our situation in the rest of the paper is that we are given a finite algebra
  $\algA$ described by a list of its elements and the tables of its (finitely
  many) basic operations, and we want to decide if there is a cube term in $\algA$. Occasionally, we shall need to distinguish the number of
  elements of $\algA$, denoted by $|A|$, from the total size of the input which
  includes the list of elements plus a table of size $|A|^r$ for each basic
  operation of $\algA$ of arity $r$. We denote the total size of $\algA$'s
  description by $|\algA|$.  

  Given a set
  $S\subset A$, we can find the subalgebra
  of $\algA$ generated by $S$ in time $O(m|\algA|)$, where $m$ is the maximum arity of a basic operation of $\algA$. The algorithm works by
  generating a sequence $S=S_0\subsetneq S_1\subsetneq S_2\subsetneq \dots$ of subsets
  of $A$ that terminates with $\Sg(S)$. We obtain $S_{i+1}$ from $S_i$ by applying all the basic operations
  of $\algA$ to tuples from $S_i$ that contain at least one element outside of
  $S_{i-1}$ (this last condition ensures that we handle each tuple at most once
  for each basic operation; in the $i=0$ step, we let $S_{-1}=\emptyset$).  We do not claim authorship of this
  algorithm; it was previously mentioned
  in~\cite{freese-valeriote-maltsev-conditions}.

  \begin{definition}
    Let $\algA$ be an algebra. A relation $R\subset A^n$ is \emph{compatible with $\algA$}
    (also called $\algA$-invariant or admissible in the literature) if $R$ is a
    subuniverse of $\algA^{n}$. In other words, $R\leq \algA^n$ if for every $m$-ary basic operation $f$ of
    $\algA$ the operation $f$ extended to $A^n$ maps $R$ into itself: For every $\vect{r_1},\dots,\vect{r_m}\in R$ we have $f(M)\in R$
    where $M$ is the $n\times m$ matrix whose $j$-th column is $\vect{r_j}$.
  \end{definition}

  The following proposition is easily proved from the definition of relations
  compatible with an algebra. We note that this proposition is part of a larger theory of
  Galois correspondence between clones of operations and relational clones
  (sets of relations closed under primitive positive
  definitions)~\cite{bodnarchuk-et-al, geiger}.
  \begin{proposition}\label{prop:Galois}
    Let $R\leq \algA^n, S\leq \algA^m$ be relations compatible with $\algA$. Then the following relations
    are also compatible with $\algA$: 
    \begin{enumerate}
      \item The unary relations $A$ and $\emptyset$,
      \item the projection of $R$ to any subset of $[n]$,
      \item the relation $R\times S$,
      \item if $n=m$, the relation $R\cap S$, and
      \item the unary relation $\{(a)\}$, where $a\in \algA$, if $\algA$ is idempotent.
    \end{enumerate}
  \end{proposition} 
  One application of Proposition~\ref{prop:Galois} that we will use is that 
  if $R$ is an $n$-ary relation compatible with an idempotent algebra $\algA$
  and $a$ is an element of $A$, then the relation $S=\{(s_1,\dots,s_{n-1})\colon (a,s_1,\dots,s_{n-1})\in R\}$ we get from $R$ by ``fixing the first
  coordinate to $a$ and projecting it out''
  is also compatible with $\algA$.


  A $(2^d-1)$-ary operation $t$ on $A$ is called a \emph{cube operation} of
  dimension $d$ if for all $\vect a,\vect b\in A^d$ we have
  \[
    t(\chi_I(\vect a,\vect b)\colon \emptyset\neq I\subset[d])=\vect a.
  \]
  The reason for this name is that if we view $2^I$ as a $d$-dimensional cube,
  then the cube term will, given a cube with one missing vertex,
	  fill in the
  empty spot. 

  A $(d+1)$-ary operation $u$ is called an \emph{edge} operation of
  dimension $d$ (or just a $d$-edge operation) if
  for any $\vect a,\vect b\in A^d$ we have
  \[
    u(\chi_I(\vect a,\vect b)\colon I=\{1,2\},\{1\},\{2\},\dots,\{d\})=\vect a.
  \]

  Finally, a $d$-ary operation $n$ is a $d$-ary near unanimity operation if for any $\vect a,\vect b\in A^d$ we have
  \[
    n(\chi_I(\vect a,\vect b)\colon I=\{1\},\{2\},\dots,\{d\})=\vect a.
  \]
  (This is not the usual way to write the near unanimity equalities, but we chose it here
  to show the similarity between near unanimity  and cube/edge operations.)

  Note that since all three of the above definitions (cube, edge, and near unanimity
  operations) require that an equality holds for any $\vect a,\vect b\in A^d$, 
  we could have also easily rewritten them as systems of identities in two variables.
  For example, a near  unanimity operation is an operation satisfying the following identities for all $x,y\in A$:
  \[f(y,x,\dots,x) = f(x,y,x,\dots,x) = \dots = f(x,\dots,x,y)=x.\]

  It is easy to see that a $d$-dimensional edge term implies the
  existence of a $d$-dimensional cube term. It turns out that one can
	  also prove the converse:
  \begin{theorem}[\protect{\cite[Theorem 2.12]{BIMMVW}}]\label{thmEdgeTerm}
    Let $\algA$ be an algebra. Then $\algA$ has a $d$-dimensional edge
    term operation if and only if $\algA$ has a $d$-dimensional cube term
    operation.
  \end{theorem}

  The situation with near unanimity is a bit more complicated since having
  a $d$-ary near unanimity operation is
  a strictly stronger condition than having an edge or a cube term. However,
  it turns out that having a $k$-ary near unanimity operation is equivalent to having a $k$-edge
  term (or, by Theorem~\ref{thmEdgeTerm}, $k$-dimensional cube term) 
  for algebras in congruence distributive varieties (for an earlier result of a similar
  flavor, see~\cite[Theorem 3.16]{sequeira-nu}).
  
  \begin{theorem}[\protect{\cite[Theorem 4.4]{BIMMVW}}]\label{thmCDedgeNU}
    For each $k\geq 3$, a variety $\mathcal V$ is congruence distributive and has a $k$-edge term
    $t$ if and only if $\mathcal V$ has a $k$-ary near unanimity term.
  \end{theorem}

  Thus any bound on the minimal dimension of a cube term is also a bound on the minimal
  arity of a near unanimity operation. 

  As noted in the Introduction, deciding the existence of a near unanimity term
  reduces to deciding the existence of a cube term: By 
  Theorem~\ref{thmCDedgeNU}, an algebra $\algA$ has a near  unanimity term if and
  only if it has a cube term and $\algA$ lies in a congruence
  distributive variety. One can test
  whether $\algA$ generates a congruence distributive variety in
  time polynomial in $|\algA|$ for idempotent algebras (see~\cite{freese-valeriote-maltsev-conditions} for the original algorithm or~\cite{barto-kazda-absorption} for a more elementary algorithm)
  and in time exponential in $\algA$ for general algebras (by
  taking the idempotent algorithm and adding the prefix $\seq{A}$
  to all tuples; see Lemma~\ref{lemMatrix} below).
 As it happens, our cube term deciding algorithms also run in polynomial time (measured in $|\algA|$) for an idempotent $\algA$ and exponential time for a general $\algA$ (see Theorem~\ref{thmIdempotentCTalg} and Corollary~\ref{cor:edge-in-exptime}). We will return to near unanimity terms in Section~\ref{secGeneralAlg}.

  To decide the existence of a cube term, we want to translate
  the problem from the language of operations into the language of relations. We have a
  good description of the shape of relations that, when compatible with an algebra
  $\algA$, prevent $\algA$ from having cube terms:

  We say that an $n$-ary relation $R$ on $A$ is \emph{elusive} if there exist
  tuples $\vect a,\vect b\in A^n$ such that $\chi_I(\vect a,\vect b)\in R$ for all 
  $I\neq\emptyset$, but $\chi_\emptyset(\vect a,\vect b)=\vect a\not\in R$.
  In this situation, the tuple $(a_1,\dots,a_n)$
  is called an \emph{elusive tuple} for $R$ (this is a notion similar to, but
  stricter than, essential
  tuples used in~\cite{Zhuk2017}; we note that elusive tuples were used, unnamed, already
  in~\cite{BIMMVW}). Elusive relations prevent the existence of cube terms of
  low dimensions:
  \begin{observation}\label{obsElusiveBlocks}
    If $\algA$ is an idempotent algebra which is compatible with an $n$-ary elusive
    relation, then $\algA$ does not have any cube term of dimension $n$ or less.
  \end{observation}
  \begin{proof}
    Let $R\leq \algA^n$ be elusive and let
    $\vect a,\vect b$ be a pair of tuples that witness the elusiveness of
    $R$.

    Since we can trivially obtain an $n$ dimensional cube term from a cube term
    of lower dimension by introducing dummy variables, we only consider the
    case that $\algA$ has a cube term of dimension $n$. Then applying this cube
    term on $(\chi_I(\vect a,\vect b)\colon I\neq\emptyset)$ would give us
    $\vect a\in R$, a contradiction.
  \end{proof}

  \begin{definition}\label{defBlocker}
  Let $\algA$ be an idempotent algebra. Then a pair $(C,D)$ is a \emph{cube term
    blocker} (or just a ``blocker'' for short) if $C$ and $D$ are nonempty subuniverses of $\algA$, $C\subsetneq D$
  and $\forall n,  D^n\setminus  (D\setminus C)^n \leq \algA^n$.
  \end{definition}
  The reason for the name ``cube term blockers'' comes from the following
  result:
 \begin{proposition}[\cite{markovic-maroti-mckenzie-blockers}]\label{propMMMblocker}
    Let $\algA$ be a finite idempotent algebra. Then $\algA$ has a cube term if
    and only if it possesses no cube term blockers.
  \end{proposition}

  Viewed in $n$-dimensional space, the relation $D^n\setminus  (D\setminus
  C)^n$ looks like a hypercube with one corner chipped off (the missing corner
  prevents cube terms from working properly). The following proposition
  gives a logically equivalent way to describe cube term blockers (note
  that the original paper~\cite{markovic-maroti-mckenzie-blockers} actually
  used this as the definition of cube term blockers and showed equivalence with
  our Definition~\ref{defBlocker}).
  \begin{proposition}[\protect{\cite[Lemma
    3.2]{markovic-maroti-mckenzie-blockers}}]\label{propBlocker}
    Let $\algA=(A;f_1,\dots,f_n)$ be an idempotent algebra and let
    $\emptyset\neq 
    C\subsetneq D$ be two subuniverses of $\algA$. Denote the arity of $f_i$ by
    $m_i$.  Then $(C,D)$ is a cube term
    blocker of $\algA$ if and only if for each $i=1,\dots,n$ there exists a
    coordinate $j$ between 1 and $m_i$ such that for each $c\in C$ and
    all $d_1,\dots,d_{m_i}\in D$ we have
    \[
      f(d_1,\dots,d_{j-1},c,d_{j+1},\dots,d_{m_i})\in C.
      \]
  \end{proposition}
 
 \begin{remark}\label{remTestBlocker}
  Proposition~\ref{propBlocker} immediately gives us an algorithm (first
  described in~\cite{markovic-maroti-mckenzie-blockers}) that decides in
  polynomial time if $(C,D)$ is a blocker for an algebra $\algA$ given by its
  idempotent basic operations: First test if $\emptyset\neq C\subsetneq D$ and $C, D\leq \algA$, then
  for every basic operation $f$ try out all the coordinates and see if
  $f(D,\dots,D,C,D,\dots,D)\subset C$ for some position of $C$. If we can find
  such a coordinate for all basic operations, then we have a blocker; else
  $(C,D)$ is not a blocker. The testing of coordinates can be implemented in time
  $O(m|\algA|)$, by passing through each table of each operation exactly once and
  keeping track of the coordinates that each tuple rules out (here $m$ is the maximum arity of a basic operation of $\algA$).
\end{remark}
  
  In the following, 
  we will need a more general version of cube term blocker.  Let $\emptyset\neq C_i\subsetneq D_i\subset A$ and $n_i\in \en$ where $i=1,2,\dots, k$. We then define
  the $(n_1+\dots+n_k)$-ary relation which we call a \emph{chipped cube} by
  \begin{align*}
    \left[\begin{matrix}
	C_1&|&D_1^{n_1}\\
	C_2&|&D_2^{n_2}\\
	   &\vdots&\\
	C_k&|&D_k^{n_k}
  \end{matrix}\right]=\left(\prod_{i=1}^k D_i^{n_i}\right)\setminus\left(
  \prod_{i=1}^k (D_i\setminus C_i)^{n_i}\right).
  \end{align*}
  The coordinates of a chipped cube naturally break down into \emph{blocks}:
  The $i$-th block, which we will often denote by $B_i$, consists of integers
  from $n_1+\dots+n_{i-1}+1$ to $n_1+\dots+n_{i-1}+n_i$ (inclusively).

\begin{remark}
  An important kind of relation employed in~\cite{kearnes-szendrei-cube-term-blockers} is a ``cross''. It turns out that crosses are exactly those chipped cubes where the sets $D_1,\dots,D_k$ are all equal to the whole universe of the algebra $\algA$.
\end{remark}
  We will sometimes omit unnecessary brackets as well as exponents equal to 1, so for example we have 
  \[
    \left[\begin{matrix}
	a&|&\{a,b\}^2\\
	c&|&c,d\\
    \end{matrix}\right]=\{a,b\}^2\times \{c,d\}\setminus \left\{\begin{pmatrix}
	b\\b\\d
    \end{pmatrix}\right\}.
  \]

  The following two observations are immediate consequences of the definition of an
  elusive relation:
  \begin{observation}\label{obsCutout}
    Let $(a_1,\dots,a_n)$ be an elusive tuple for some relation $R$ compatible
    with the algebra $\algA$. Then there exist elements $b_1,\dots,b_n$
    such that $(a_1,\dots,a_n)$ is an elusive tuple for the relation
    \[
    S=\Sg\left(\left[\begin{matrix}
	b_1|a_1,b_1\\
	b_2|a_2,b_2\\
	      \vdots\\
	b_n|a_n,b_n
    \end{matrix}\right]\right),
    \]
    and $S\subset R$.
  \end{observation}
\begin{observation}\label{obsCompleteCross}
  Let $\algA$ be an algebra. Let 
  \[
    E= 
    \left[\begin{matrix}
	C_1&|&D_1^{n_1}\\
	C_2&|&D_2^{n_2}\\
	   &\vdots&\\
	C_k&|&D_k^{n_k}
  \end{matrix}\right]
  \]
  be a chipped cube where each $D_1,\dots,D_k$ is a subuniverse of $\algA$ and
  assume that the relation $\Sg(E)$ is \emph{not} elusive. Then 
  \[
	  \Sg(E)=D_1^{n_1}\times D_2^{n_2}\times\dots\times D_k^{n_k}.
  \]
  \end{observation}

  \section{General cube term results}
  The following lemma is a variant of several previously known versions of the local-to-global arguments for idempotent algebras, see~\cite[Theorem 2.8]{markovic-maroti-mckenzie-blockers} and~\cite{horowitz-ijac}; the first condition is the global and the second the local one. The ``prefix'' $\seq A$ serves here to let us argue about general algebras as if they were idempotent.
  
  The key observation to use the previous local-to-global results for idempotent algebras is to note that the tuples of
  $\Sg(\{\vect u\colon \vect u=\seq{A}\vect
	  \chi_I(\vect a,\vect b),\, I\neq\emptyset\})$ that begin with $\seq{A}$ are generated using the idempotent reduct of $\algA$.
  
   \begin{lemma}\label{lemMatrix}
    Let $\algA$ be a finite algebra. Then the following are equivalent:
    \begin{enumerate}
      \item\label{itmDcube} $\algA$ has a cube term of dimension $d$,
      \item\label{itmMatrix} For all $\vect a,\vect b\in A^d$
	we have
	\[
	  \seq{A}\vect a\in\Sg(\{\vect u\colon \vect u=\seq{A}\vect
	  \chi_I(\vect a,\vect b),\, I\neq\emptyset\}).
\]
    \end{enumerate}
  \end{lemma}
  \begin{proof}
  By applying a cube term on the subalgebra in~\ref{itmMatrix} it is immediate to see that (\ref{itmDcube}) implies (\ref{itmMatrix}).
  
  For the other implication, we essentially just cite~\cite{horowitz-ijac}. Let $\algA_{id}$ be the idempotent reduct of $\algA$ -- the algebra on $A$ whose basic operations are all the idempotent operations from the clone of $\algA$. It is easy to see that the condition (\ref{itmMatrix}) is equivalent to 
  \begin{enumerate}
      \item[(3)]\label{itmThree} For all $\vect a,\vect b\in A^d$
	we have
	$
	  \vect a\in\Sg_{\algA_{id}}(\{\chi_I(\vect a,\vect b)\colon I\neq\emptyset\}).$
	  \end{enumerate}
	  
	  Let $E$ be the $d\times (2^{d}-1)$ matrix with columns indexed by nonempty subsets of $[n]$ whose $I$-th column is $\chi_I(y,x)$ for variable symbols $x,y$. Then the condition~(\ref{itmThree}) on $\algA_{id}$ is what~\cite[Definition 2.2]{horowitz-ijac} calls having all local $E$-special cube terms on sets of size $d$. It is easy to verify that our $E$ satisfies the DCP~\cite[Definition 2.5]{horowitz-ijac} and hence using~\cite[Theorem 2.6]{horowitz-ijac} the matrix $E$ has the local-to-global property of size $d$~\cite[Definition 2.3]{horowitz-ijac} and hence if $\algA_{id}$ satisfies (\ref{itmThree}) (or~\ref{itmMatrix}) then $\algA_{id}$ has a cube term. Since the clone of operations of $\algA_{id}$ is a subset of the clone of operations of $\algA$, it follows that $\algA$ also has a cube term.
  \end{proof}

The following lemma sheds some light on what minimal compatible elusive relations look
like. The additional assumption that we are dealing with a
chipped cube will be justified later in Lemma~\ref{lemCCbe}.

  \begin{lemma}\label{lemCubeNotes}
    Let $\algA$ be an algebra, $R\subset A^n$ be inclusion minimal among elusive
    relations compatible with $\algA$ (that is, no proper subset of $R$ is both an elusive relation and compatible with $\algA$). Assume moreover that $R$ is equal to the
    chipped cube
    \[
\left[\begin{matrix}
	      C_1&|&D_1\\
		 &\vdots\\
	      C_{n}&|&D_{n}\\
	\end{matrix}\right].
      \]
Let $\vect a=(a_1,\dots,a_n)$ and $\vect b=(b_1,\dots,b_n)$
    be two tuples witnessing the elusiveness of $R$. 
      Then:
      \begin{enumerate}[(a)]
	\item \label{itmGenerators}$R=\Sg(\{\chi_I(\vect a,\vect b)\colon \emptyset\neq I\subset[n]\})$,
  \item\label{itmTight} there is no subuniverse $E$ of $\algA$ such that $C_i\subsetneq
    E\subsetneq D_i$ for some $i$; in particular $D_i=\Sg(a_i,b_i)$ for each
    $i$,
	\item \label{itmSymmetry}if $i,j$ are such that $(a_i,b_i)=(a_j,b_j)$, then $(C_i,D_i)=(C_j,D_j)$,
	\item \label{itmIntersect} if $D_i=D_j$ for some $i$ and $j$ then
	  either $C_i=C_j$ or $C_i\cap
	  C_j=\emptyset$.
      \end{enumerate}
   \end{lemma}
  \begin{proof}
    Note that for each $i$ we have $a_i\in D_i\setminus C_i$ and $b_i\in C_i$.

    Part~(\ref{itmGenerators}) follows from the fact that $\Sg(\{\chi_I(\vect a,\vect b)\colon
    \emptyset\neq I\subset[n]\})$ is the smallest compatible relation that
    contains all the tuples that witness the elusiveness of $R$.

    The proof of point~(\ref{itmTight}) is similar. Were there $E$ strictly between $C_i$ and
    $D_i$, we could restrict the $i$-th coordinate of $R$ to $E$ and obtain a
    smaller elusive compatible relation, proving~(\ref{itmTight}). 

    To see that~(\ref{itmSymmetry}) holds, take $i$ and $j$ such that $a_i=a_j$ and
    $b_i=b_j$. Without loss of generality let $i<j$. Then the set of generators of $R$ is invariant under the
    permutation that swaps $i$-th and $j$-th coordinates. Therefore, $R$ is
    invariant under such a permutation of coordinates as well. Consider now
    \[
      R'=\{(e,f)\colon
      (a_1,\dots,a_{i-1},e,a_{i+1},\dots,a_{j-1},f,a_{j+1},\dots,a_n)\in R\}.
      \]
    Since $R$ is a chipped cube, it follows that $R'=C_i\times D_j\cup
    D_i\times C_j$ and from the symmetry of $R$, we get that $R'$ is symmetric
    as a binary relation. It is straightforward to verify that this can only
    happen when $C_i=C_j$ and $D_i=D_j$.

    To prove part~(\ref{itmIntersect}), assume (without loss of generality) that $D_1=D_2=D$ and there exist $e\in C_1\setminus C_2$ and $c\in C_1\cap C_2$. 
    
    Consider then the chipped cube we get from $R$ by
    switching the first two coordinates:
\[R'=
\left[\begin{matrix}
    	      C_2&|&D\\
	      C_1&|&D\\
              C_3&|&D_3\\
		 &\vdots\\
	      C_{n}&|&D_{n}\\
	\end{matrix}\right].
      \]
    By symmetry, $R'$ is a relation compatible with $\algA$. This makes $R\cap R'$ also compatible with $\algA$. We now claim that $R\cap R'$ is an elusive relation that is strictly smaller than $R$. To see $R\cap R'\subsetneq R$, choose  $d_i\in D_i\setminus C_i$ for $i=2,\dots,n$ and observe that $R$ contains the tuple $ed_2d_3\cdots d_n$, while $R\cap R'$ does not.
    To prove that $R\cap R'$ is elusive, choose $c_i\in C_i$ for $i=3,4,\dots,n$ and observe that $R\cap R'$ contains the tuple $\chi_I(e d_2 d_3\cdots d_n,c c c_3\cdots c_n)$ for any $I\neq \emptyset$, but not for $I=\emptyset$. This concludes the proof of the last point.
  \end{proof}
In the proof above we made progress thanks to  swapping two coordinates of $R$.
Later in the paper, we will be working with a general mapping that moves
coordinates of tuples around.

  \section{Cube terms in idempotent algebras}\label{secCTBidemp}
  
  \begin{lemma}\label{lemCCbe}
    Let $\algA$ be an idempotent algebra, $R$ a relation that is inclusion minimal among elusive relations compatible with $\algA$. Then $R$ is a chipped cube.
  \end{lemma}

  \begin{proof}
    Given that $R$ is minimal, our strategy will be to fit a \emph{maximal} chipped cube 
     into $R$ and show that this chipped cube is equal to $R$.

     Let $(a_1,\dots,a_n)$ be an elusive tuple for $R$.
     Let moreover 
\[
  E=  \left[\begin{matrix}
	      C_1&|&D_1\\
		 &\vdots\\
	      C_{n}&|&D_{n}\\
	\end{matrix}\right]
\]
     be an inclusion-maximal chipped cube such that (1) $E\subset R$ and (2) $a_i\in D_i$ 
     for each  $i=1,2,\dots,n$. (At least one such chipped cube exists by
     Observation~\ref{obsCutout}.)

     We prepare ground for our proof by exploring properties of the sets $D_i$.
     From the maximality of $E$, it follows that  
     each of the sets $D_i$ is a subuniverse of $\algA$: Were, say, $u\in
     \Sg(D_1)\setminus D_1$ then $u=t(d_1,\dots,d_m)$ for some operation $t$ of
     $\algA$ and suitable $d_1,\dots,d_m\in D_1$. Take any $e_2,\dots,e_n\in A$ such that
     for each $i$ we have $e_i\in D_i$ and for at least one $i\in\{2,\dots,n\}$ we have
     $e_i\in C_i$. By definition of a chipped cube, we have $(d_j,e_2,\dots,e_n)\in E$ for each $j=1,\dots,m$; applying $t$ thus gives 
     us $(u,e_2,\dots,e_n)\in \Sg(E)$. Therefore $\Sg(E)\subset R$ 
     contains the chipped cube
\[
  F=  \left[\begin{matrix}
      C_1&|&D_1\cup \{u\}\\
      C_2&|&D_2\\
		 &\vdots\\
	      C_{n}&|&D_{n}\\
	\end{matrix}\right],
\]
a contradiction with the maximality of $E$.

    From the minimality of $R$, we immediately get $R=\Sg(E)$ and what is more
    $\pi_i(R)=\Sg(C_i\cup\{a_i\})$ for each $i$. The latter equality
     follows from the fact that
     $R\cap \prod_{i=1}^n \Sg(C_i\cup\{a_i\})$ is an elusive relation compatible with $\algA$ and 
     we chose $R$ to be a minimal elusive relation compatible with $\algA$.
     Since for each $i$ we have $\Sg(C_i\cup\{a_i\})\subset D_i \subset
     \pi_i(R)$ and the outer pair of sets is equal, it follows that actually
     $D_i=\Sg(C_i\cup \{a_i\})=\pi_i(R)$ for each $i$.

     We are now ready to show that $E=R$. Assume otherwise and 
     choose $\vect u\in R\setminus E$ that
     agrees with $(a_1,a_2,\dots,a_n)$ on as many coordinates as possible.
     Up to reordering of coordinates, we thus have a tuple $\vect u=(u_1,u_2,\dots,u_n)\in
     R\setminus E$ such that $(a_1,u_2,\dots,u_n)\not\in R$.
     
     Since $\vect u\in R\setminus E$, we see that
     $u_i\in D_i\setminus C_i$ for all $i=1,\dots,n$. We will show that $R$ contains the
     chipped cube
\[
      \left[\begin{matrix}
              C_1\cup \{u_1\}&|&D_1\\
              C_2&|&D_2\\
              C_3&|&D_3\\
	 &\vdots\\
	      C_{n}&|&D_{n}\\
	\end{matrix}\right],
\]
yielding a contradiction with the maximality of $E$. To prove this, we need to show that 
$\{u_1\}\times D_2\times\dots\times D_n\subset R$.

 Observe first that
$\Sg(C_1\cup\{u_1\})\times\{(u_2,\dots,u_n)\}\subset R.$
   This follows from the inclusion $(C_1\cup\{u_1\})\times\{(u_2,\dots,u_n)\}\subset R$ and
   the idempotence of $\algA$.

   Let $D'=\Sg(C_1\cup \{u_1\})$. Since $(a_1,u_2,\dots,u_n)\not\in R$,
   we obtain $a_1\not\in D'$, making $D'$ strictly smaller than $D_1$.

     Consider now the relation $R'=R\cap (D'\times D_2\times D_3\times \dots\times
     D_n)$. From $R'\subsetneq
R$ we see that $R'$ is not elusive. However, $R'$ contains the chipped cube
     \[
      \left[\begin{matrix}
              C_1&|&D'\\
              C_2&|&D_2\\
              C_3&|&D_3\\
	 &\vdots\\
	      C_{n}&|&D_{n}\\
	\end{matrix}\right].
\]
This is only possible if $\{u_1\}\times D_2\times\dots\times
D_n\subset R'\subset R$, which is exactly what we needed.
  \end{proof}
We would like to remark that Lemma~\ref{lemCCbe} does not hold for general algebras. As an example, consider the algebra $\algA$ on the set $\{0,1,2\}$ with one unary constant operation $c_2(x)=2$. For any $n\geq 2$ let 
\[
R=\{2^{\bf n}\}\cup \{0,1\}^n\setminus\{10^{\bf n-1}\}.
\]
This relation is compatible with $\algA$ and elusive. It is straightforward to verify that $R$ is also minimal such ($n$-ary elusive relation needs to contain at least $2^n-1$ tuples; $R$ is just one tuple larger than this theoretical minimum, and a case consideration shows that we can't discard any tuple from $R$). However, $R$ is not a chipped cube: The projections of $R$ to each coordinate are all equal to $\{0,1,2\}$ and $R$ contains the tuple $2^{\bf n}$. It is not hard to show that any chipped cube with these two properties contains at least $3^{n-1}$ tuples with at least one entry 2, a property that $R$ fails to have.

  
\begin{corollary}\label{corIdempotentObstruction}
  Let $\algA$ be a finite idempotent algebra, and let $d\geq 2$ be an integer. Then $\algA$ has a cube term of dimension $d$ if and only if there is no $d$-ary chipped cube relation compatible with $\algA$.
  \end{corollary}
  \begin{proof}
  Cube terms are incompatible with chipped cubes, so the interesting implication is that the absence of a $d$-dimensional cube term gives us a $d$-ary compatible chipped cube. 
  
  Assume thus that $\algA$ has no $d$-dimensional cube term. We claim that then $\algA$ is compatible with some $d$-ary elusive relation. By Lemma~\ref{lemMatrix}, there exist $\vect a,\vect b\in A^d$ such that $\seq{A}\vect a\not\in \Sg(\{\seq{A}\chi_I(\vect a,\vect b)\colon I\neq \emptyset\})$. Since $\algA$ is idempotent, we can remove the $\seq{A}$ prefix and have $\vect a\not\in \Sg(\{\chi_I(\vect a,\vect b) \colon I\neq \emptyset\})$. This amounts to saying that $\vect a$ is an elusive tuple for $\Sg(\{\chi_I(\vect a,\vect b) \colon I\neq \emptyset\})$. To finish the proof, we take an inclusion minimal $d$-ary elusive relation $E$ compatible with $\algA$. By Lemma~\ref{lemCCbe}, this $E$ is a chipped cube.
  \end{proof}
  
  \smallskip
  The following Lemma is a nontrivial consequence of
  Proposition~\ref{propBlocker}.
  \begin{lemma}
    Let $\algA=(A;f_1,\dots,f_\ell)$ be an idempotent algebra that has a cube
    term. Denote by $m_1,\dots,m_\ell$ the arities of $f_1,\dots,f_\ell$.
    Assume that a chipped cube
  \[
    F= 
    \left[\begin{matrix}
	C_1&|&D_1\\
	C_2&|&D_2\\
	   &\vdots&\\
	C_k&|&D_k
  \end{matrix}\right]
  \]
is compatible with $\algA$. Let
    \[
      U_i=\{j\colon \text{$(C_j,D_j)$ is \emph{not} a cube term blocker in the algebra $(A;f_i)$}\}
    \]
  where $i$ goes from 1 to $\ell$ (see Definition~\ref{defBlocker} for what a
    cube term blocker is). Then:
    \begin{enumerate}[(a)]
      \item $\bigcup_{i=1}^\ell U_i=[k]$, and\label{partBigcup}
      \item  for each $i=1,\dots,\ell$ we have $m_i\geq 1+|U_i|$.
    \end{enumerate}
    \end{lemma}
    \begin{proof}
    Note that we implicitly have $\emptyset\neq C_i\subsetneq D_i$ for each
      $i$ from the definition of a chipped cube.

      Since $\algA$ has a cube term, there is no cube term blocker in $\algA$. In
    particular none of $(C_1,D_1)$, $(C_2,D_2)$, \dots, $(C_k,D_k)$ are
      blockers for $(A;f_1,\dots,f_\ell)$. Were, say, $(C_1,D_1)$ a blocker for
      all algebras $(A;f_i)$ then each $f_i$ maps each relation $D_1^n\setminus(D_1\setminus C_1)^n$ into itself, making $(C_1,D_1)$ a blocker for the whole algebra $\algA$.
      Therefore, for each $j\in [k]$ there exists an $i$ so
      that $j\in U_i$, i.e. $\bigcup_{i=1}^\ell U_i=[k]$, giving us
      part~(\ref{partBigcup}).
      
    It remains to show that for each $i\in[\ell]$ we have
      $m_i\geq 1+|U_i|$. Assume for a contradiction that $m_i\leq |U_i|$ for some $i$. 
      Without loss of generality we can assume that in fact $[m_1]\subseteq U_{1}$
      (we are free to reorder the $f_i$'s and $(C_j,D_j)$'s).
    We then consider the chipped cube
    \[
      E=
    \left[\begin{matrix}
	C_1&|&D_1\\
	C_2&|&D_2\\
	   &\vdots&\\
	C_{m_1}&|&D_{m_1}\\
     \end{matrix}\right].
      \]
    It is easy to see that $E$ can be obtained from $F$ by restricting all but
    the first $m_1$ coordinates to some singleton values from $D_j\setminus C_j$ 
      and projecting the result to the first $m_1$ coordinates (here we need that
      $\algA$ is idempotent). Therefore $f_1$ 
    preserves $E$.
   
      However, we know that $(C_1,D_1),\dots,(C_{m_1},D_{m_1})$ are not blockers for
      $(A;f_1)$. By Proposition~\ref{propBlocker},
         for each $j\in \{1,2,\dots,m_1\}$ we then can find a tuple
      \[
	(a_{j,1},a_{j,2},\dots,a_{j,m_1})\in D_j^{m_1}
	\]
	such that $a_{j,j}\in C_j$ and
    $f_1(a_{j,1},\dots,a_{j,m_1})\not\in C_j$. 

    Arrange the above mentioned tuples into rows of an $m_1\times m_1$ matrix
      $M$. Since $a_{j,j}\in C_j$ for all $j=1,\dots,m_1$, 
    each column of $M$ belongs to $E$. Therefore, we should have $f_1(M)\in E$
    as well. But $f_1$ applied to the $j$-th row of
    the matrix $M$ gives us an element from $D_j\setminus C_j$ for each $j$, so
    $f_1(M)$ fails to be in $E$, a contradiction.
    
  \end{proof} 
  Since the sets $U_i$ in the above theorem depend only on $f_i$ and the sets $C_j,D_j$, we can
  generalize the result to the case when $(C_j,D_j)$ appears multiple times in
  $F$:
  \begin{corollary}\label{corCounting}
   Let $\algA=(A;f_1,\dots,f_\ell)$ be an idempotent algebra that has a cube
    term. Denote by $m_1,\dots,m_\ell$ the arities of $f_1,\dots,f_\ell$.
    Assume that a chipped cube
  \[
    F= 
    \left[\begin{matrix}
	C_1&|&D_1^{n_1}\\
	C_2&|&D_2^{n_2}\\
	   &\vdots&\\
	C_k&|&D_k^{n_k}
  \end{matrix}\right]
  \]
   is compatible with $\algA$. Then there exists a family of sets
    $U_1,\dots,U_\ell$ such that $\bigcup_{i=1}^\ell U_i=[k]$ and for each $i$
    we have $m_i\geq 1+\sum_{j\in U_i}n_j$.
  \end{corollary}

  We are now ready to give a lower bound on the dimension of a cube term in finite idempotent algebras with cube terms. For a version of the following theorem that works for infinite idempotent algebras, see~\cite[Theorem 4.1]{kearnes-szendrei-cube-term-blockers}
  (we discuss the relationship between our paper
  and~\cite{kearnes-szendrei-cube-term-blockers} in detail at the end of the Introduction).
  \begin{theorem}\label{thmIdempotentCTB} 
    Let $\algA=(A;f_1,\dots,f_\ell)$ be a finite idempotent algebra. Let
    $m_1\geq m_2\geq \dots\geq m_\ell$ be the arities of the basic operations of $\algA$. 
    Let $N=1+\sum_{i=1}^r (m_i-1)$ where $r=\min(\ell,{\binom{|A|}{2}})$. Assume
    that $N>2$ or $|A|>2$. Then the following are equivalent:
    \begin{enumerate}[(a)]
      \item $\algA$ does not have a cube term.\label{itmIdmpNoCube}
      \item $\algA$ has a cube term blocker.\label{itmCTB}
      \item $\algA$ does not have a cube term of dimension $N$. \label{itmIdmpNoHugeCube}
      \item There is an $N$-ary elusive relation compatible with $\algA$.
	\label{itmNcrit} 
      \item There exists an $N$-ary chipped cube compatible with $\algA$.\label{itmNchipped cube}
    \end{enumerate}
  \end{theorem}
  \begin{proof}
    The equivalence of~(\ref{itmIdmpNoCube}) and~(\ref{itmCTB}) is already
    known~\cite{BIMMVW}.

    If~(\ref{itmCTB}) holds and $(C,D)$ is a cube term blocker for $\algA$, then
    the relation $D^N\setminus (D\setminus C)^N$ is an $N$-ary chipped cube,
    proving~(\ref{itmNchipped cube}).
    
    From~(\ref{itmNchipped cube}), condition~(\ref{itmNcrit}) follows trivially, from which~(\ref{itmIdmpNoHugeCube}) follows immediately by Observation~\ref{obsElusiveBlocks}.
   Corollary~\ref{corIdempotentObstruction} gives us that~(\ref{itmIdmpNoHugeCube})
    implies~(\ref{itmNchipped cube}). 
    
    It remains to show that~(\ref{itmNchipped cube})
    implies~(\ref{itmIdmpNoCube}). We proceed by contradiction, assuming that
    $\algA$ has both a cube term and an $N$-ary compatible chipped cube. Take $n$ smallest such that
    $\algA$ has a cube term of dimension $n+1$. Certainly $n\geq N$ since chipped cubes are elusive relations.
    
    Since $\algA$ is idempotent without an $n$-dimensional cube term, there must exist an $n$-ary elusive relation $R$ compatible with $\algA$. Let $\vect a,\vect b$ be the tuples that witness the elusiveness of $R$. We choose $R$ to be inclusion minimal; by Lemma~\ref{lemCCbe} $R$ is a chipped cube. Applying Lemma~\ref{lemCubeNotes}, we get that 
    \[
      E=\left[\begin{matrix}
	      C_1&|&D_1\\
		 &\vdots\\
	      C_{n}&|&D_{n}\\
	\end{matrix}\right],
      \]
      where $D_i=\Sg(a_i,b_i)$ for all $i\in[n]$ and 
      $(C_i,D_i)=(C_j,D_j)$ whenever $(a_i,b_i)=(a_j,b_j)$.
    Let us reorder the coordinates of $\vect a,\vect b$ so that
    identical pairs $(a_i,b_i)$ are grouped together. Let $k$ be the number of
    distinct pairs $(a_i,b_i)$ and denote by $n_i$ the number of times the
    $i$-th pair appears. After this reordering (and renaming of $C_i,D_i$'s) we have
     \[
      E=\left[\begin{matrix}
	  C_1&|&D_1^{n_1}\\
		 &\vdots\\
	  C_{k}&|&D_{k}^{n_k}\\
	\end{matrix}\right],
      \]
    where of course $n_1+n_2+\dots+n_k=n$. Since $a_i\neq b_i$ for each $i$, we
    obtain $k\leq |A|(|A|-1)$.  
    
    Assume first that $N>2$ (we will deal with the special cases $N=2$ and $N=1$ later). We claim that then in fact $k\leq |A|(|A|-1)/2$. To
    prove this, we show that for all $i,j$ we have  $(a_i,b_i)\neq (b_j,a_j)$.
    For  $i=j$, this is obvious, so assume without loss of generality that
    $i=1,j=n$. Then $D_1=D_n=\Sg(a_1,a_n)$ and $C_1\cap
    C_n=\emptyset$ by Lemma~\ref{lemCubeNotes}. Consider the relation
    \[
      F=\{(x_1,\dots,x_{2n-2})\colon\exists z\in D_1,
	(x_1,\dots,x_{n-1},z),(z,x_n,\dots,x_{2n-2})\in
      E\}.
    \]
    From the definition of $F$ we see that it is  compatible with $\algA$. We claim that
    the tuple $(a_1,a_2,\dots, a_{n-1},a_2,a_3,\dots, a_n)$ is elusive for $F$. This tuple is not a member of $F$ because were 
    $(a_1,\dots, a_{n-1},a_2,\dots, a_n)\in F$, then there would exist a $z$ such
    that $(a_1,\dots,a_{n-1},z),(z,a_2,\dots,a_n)\in E$. However, such $z$
    would need to lie in both $C_1$ and $C_n$, a contradiction. On the other
    hand, if we rewrite one or more entries of
    $(a_1,\dots,a_{n-1},a_2,\dots,a_n)$ to $b_i$, we can choose $z=b_n$ or
    $z=b_1$ and satisfy both $(x_1,\dots,x_{n-1},z)\in E$ and
    $(z,x_2,\dots,x_n)\in E$.

    Since $n\geq N>2$, we have $2n-2>n$, and so $F$ is an elusive compatible
    relation of arity higher than $n$, a contradiction with $\algA$ having
    an $(n+1)$-dimensional cube term. Therefore $k\leq
    |A|(|A|-1)/2$.

    We now apply Corollary~\ref{corCounting} to the relation
 \[
      E=\left[\begin{matrix}
	  C_1&|&D_1^{n_1}\\
		 &\vdots\\
	  C_{k}&|&D_{k}^{n_k}\\
	\end{matrix}\right],
      \]
      obtaining a family of sets $U_1,\dots,U_\ell$ (one for each basic
      operation of $\algA$) such that $U_1,\dots,U_\ell$ cover $[k]$ and
      $m_i-1\geq\sum_{j\in U_i} n_j$ for each $i$. Take
      $r=\min(\ell,{\binom{|A|}{2}})$ and observe that we can then cover $[k]$ by
      at most $r$ sets chosen from $U_1,\dots,U_\ell$. (If $r={\binom{|A|}{2}}<\ell$, the existence of such a covering of $[k]$ follows from $k\leq \binom{|A|}{2}$.) 
      Let $I\subset [\ell]$ be a set of indices of size $r$ so 
      that $\bigcup_{i\in I} U_i=[k]$. Since the
      arities $m_1,\dots,m_\ell$ are at least one and ordered in a
      nonincreasing order, we get the inequalities:
      \begin{align*}
	\sum_{i=1}^r (m_i-1)\geq \sum_{i\in I} (m_i-1)\geq \sum_{j=1}^k n_j=n.
      \end{align*}
      Therefore, $N-1= \sum_{i=1}^r (m_i-1)\geq n$, a contradiction with the
      assumption $n+1>N$.

      Assume now that $N=2$. Given the formula for $N$, we must have $m_1=2$
      and $m_2=\dots=m_\ell=1$. (Note that we are using $|A|>2$
      here; we need $r\geq 2$ to make sure that $m_2=1$.) Since $\algA$ is
      idempotent, the operations $f_2,\dots,f_\ell$ are unary identity
      mappings -- without loss of generality let $\ell=1$. We again take
      $n\ge N=2$ such that $n+1$ is the least dimension of a cube term in $\algA$ and  construct the chipped cube $E$ as above. We only have the $k\leq
      |A|(|A|-1)$ bound in this case; however, we can finish the proof anyway:
      We apply Corollary~\ref{corCounting} with $\ell=1$ and $m_1=2$ to get
      that $2=m_1\geq 1+n_1+\dots+n_k=1+n$, i.e.  $1\geq n$. This is a
      contradiction with $n\geq N=2$, finishing the proof. 

     Finally, if $N=1$, the clone of operations of $\algA$ consists of
     projections only and both~(\ref{itmNchipped cube}) and~(\ref{itmIdmpNoCube}) are trivially true.
  \end{proof}
  Using Theorem~\ref{thmIdempotentCTB}, we can obtain a bound on the dimension of a cube term that depends quadratically on $|A|$ and linearly on the maximum arity of a basic operation of $\algA$.
\begin{corollary}\label{cor:quadratic-linear}
Let $\algA$ be finite idempotent algebra whose basic operations have a maximal arity $m$. Then $\algA$ has a cube term if and only if $\algA$ has a cube term of dimension at most $1+(m-1)\binom{|A|}{2}$.
\end{corollary}
\begin{proof}
By Theorem~\ref{thmIdempotentCTB}, if $\algA$ has a cube term blocker, it has a cube term blocker of arity at most
$N=1+\sum_{i=1}^r(m_i-1)$ where each $m_i\leq m$ and $r\leq \binom{|A|}{2}$. Hence, we can estimate the arity of the cube term as $N\leq 1+(m-1)\binom{|A|}{2}$ and the claim follows.
\end{proof}

  It turns out that the bound on the dimension of a cube term in
  Theorem~\ref{thmIdempotentCTB} and Corollary \ref{cor:quadratic-linear} is tight (see also~\cite[Example
  4.4]{kearnes-szendrei-cube-term-blockers}):
  \begin{theorem}\label{thmIdempotentTight}
    Let $\ell,n\in\en$ and $m_1\geq m_2\geq\dots\geq m_\ell$ be positive integers. 
    Let $N=1+\sum_{i=1}^r (m_i-1)$ where $r=\min(\ell,\binom{n}{2})$. Assume
    that either $n>2$ or $n=2$ and $N>2$. Then there exist idempotent operations $f_1,\dots,f_\ell$ on the
    set $[n]$ of arities $m_1,\dots,m_\ell$ such that $\algA=([n];f_1,\dots,f_\ell)$ has a
    cube term of dimension $N$, but no cube term of dimension $N-1$.
  \end{theorem}
  \begin{proof}
    Without loss of generality assume that all $m_i$ are at least 2 (unary idempotent
    operations are equal to the identity mapping and therefore not interesting).
    
    Let us handle the case $n>2$ and $N=2$ first. Since there is no cube
    term of dimension $N-1=1$, all we need to do is produce an algebra on $n$
    elements with one basic idempotent operation of arity 2 and a Maltsev term
    (i.e. a cube term of dimension 2).  We choose $\algA=([n],f)$ to be an
    idempotent quasigroup of order $n$.  Such a quasigroup exists for all $n>2$
    (see~\cite[Theorem 2.2.3]{design-theory}) and all quasigroups have a
    Maltsev operation. 

    Assume now that $N>2$ and $n\geq 2$.
    Partition the set $\{(a,b)\in [n]^2\colon
    a<b\}$ into $r$ (nonempty, disjoint) sets $J_1,\dots,J_r$ (such
    a partition will exist because $r\leq \binom{n}{2}$). We then define the operations $f_1,\dots,f_r$ as
    follows: For each $i=1,2,\dots,r$, let
    \[
      f_i(a,a,\dots,a,b)=\dots=f_i(b,a,\dots,a,a)=a
    \]
  for all pairs $(a,b)\in J_i$. Otherwise, let $f_i(x_1,\dots,x_{m_1})=\max(x_1,\dots,x_{m_1})$. If $\ell>r$, we choose the remaining operations $f_{r+1},\dots,f_\ell$ 
    to be projections to the first coordinate.

    We now claim that the algebra $\algA=([n];f_1,\dots,f_\ell)$ has an
    $N$-dimensional  cube term, but no $(N-1)$-dimensional cube term. 
    
    To prove
    that $\algA$ has an $N$-dimensional cube term, it is enough to show that
    there is no cube term blocker in $\algA$ and apply
    Theorem~\ref{thmIdempotentCTB}.
    Let  $\emptyset\neq C\subsetneq
    D\subset A$ be a candidate for a blocker. Pick a pair $c\in C$, $d\in D\setminus C$. We will show that there is an $i$ such that
    \[
      f_i(c,d,\dots,d,d)=f_i(d,c,\dots,d,d)=\dots=f_i(d,d,\dots,d,c)=d\not\in C,
    \] 
  which contradicts Proposition~\ref{propBlocker}.
  
    If $c<d$, we do the following. 
    If the map $f_1$ is at least ternary then
    \[
      f_1(c,d,\dots,d,d)=f_1(d,c,\dots,d,d)=\dots=f_1(d,d,\dots,d,c)=\max(c,d)=d.
    \] 
    If $m_1 =2$, then $m_2=2$ (because of $N>2$) 
    and for some $j\in\{1,2\}$ we have 
    $f_{j}(c,d) = f_{j}(d,c) = \max(c,d)=d$.
    
    If $c>d$, then there exists an $i$ such that $(d,c)\in J_i$ and so from the
    definition of $f_i$'s we 
    have 
    \[
      f_i(c,d,\dots,d,d)=f_i(d,c,\dots,d,d)=\dots=f_i(d,d,\dots,d,c)=d.
    \] 

    It remains to show that $\algA$ has no cube term of dimension $N-1$. We do
    this by
    constructing an $(N-1)$-ary elusive relation that is compatible with $\algA$. Pick a pair
    $(a_i,b_i)\in J_i$ for every $i=1,2,\dots,r$ and consider the $N$-tuples
    $\vect a=a_1^{\bf m_1-1}a_2^{\bf m_2-1}\dots a_r^{\bf m_r-1}$ and
    $\vect b=b_1^{\bf m_1-1}b_2^{\bf m_2-1}\dots b_r^{\bf m_r-1}$. We claim that all
    operations $f_1,\dots,f_r$ map
    $R=\{\chi_I(\vect a,\vect b)\colon I\neq \emptyset\}$
  into itself. This will conclude the proof, since $R$ is elusive
    (as witnessed by $\vect a\not\in R$).
    
    We will show that $f_1$ maps $R$ into itself; the proof for
    $f_2,\dots,f_r$ is similar. Since $f_1$ is conservative (it always returns one of
    its arguments) we only have to show that there is no $(N-1)\times m_1$
    matrix $M$ with columns from $R$ such that $f_1(M)=\vect a$. Since the
    first $m_1-1$ entries of $\vect a$ are $a_1$ and $f_1$ only returns
    $a_1$ when at most one of inputs differs from $a_1$, there can be
    at most $m_1-1$ entries different from $a_1$ in the top $(m_1-1)\times m_1$
    submatrix of $M$. Since $M$ has $m_1$ columns there is a column $\vect c$ of $M$ that
    begins with $a_1^{\bf m_1}$. Let us examine this column $\vect c$ more
    closely.
    
    Since $\vect c\in R$, there has to be $i>1$ such
    that the $i$-th block of $\vect c$ contains $b_i$.
    Since $b_i>a_i$ and $(a_i,b_i)\not \in J_1$ (this follows from $J_i\cap J_1\neq
    \emptyset$), $f_1$ applied to $b_i$ and any combination of $a_i$'s and
    $b_i$'s will return $b_i$. Therefore, the $i$-th block of $f_1(M)$ also
    contains $b_i$, yielding $f_1(M)\neq \vect a$ as we needed.
    
  \end{proof}
  What happens to the claims of Theorem~\ref{thmIdempotentCTB} when $|A|,N\leq 2$?  If $N=1$ or $|A|=1$ the algebra $\algA$ is
  not very interesting since it contains only projections. When $|A|=N=2$, we
  can have algebras that have a cube term, but no cube term of dimension 2:
  Consider $\algA=(\{0,1\},\wedge,\vee)$. This algebra is compatible with the binary elusive relation $\{(0,0),(0,1),(1,1)\}$ (elusive tuple $(1,0)$) and so does not have a cube term of dimension 2 
  (known as the Maltsev operation). However, $\algA$ has a ternary near unanimity term
  and thus a cube term of dimension 3.


  \section{Deciding cube terms in the idempotent case}
  In this section, we describe a polynomial time algorithm that
  decides whether the input idempotent algebra $\algA$ has a cube term.. What is more, if $\algA$ has no cube term, the algorithm will prove this by outputting a cube term blocker for $\algA$.

  \begin{algorithm}[t]
    \KwData{Idempotent algebra $\algA$ given by tables of its basic operations}
    \KwResult{A cube term blocker $(C,D)$ if $\algA$ has one, ``No''
    otherwise.}
    \For{$c\in A$\label{line:for}} {
      Let $S:=\{c\}$\;
	\While {$S\neq A$\label{line:while}}{
	Choose a $d\not\in S$ so that $\Sg(c,d)$ is \label{line:choose} (inclusion)
	  minimal among all such choices of $d$.\label{stepInnerFirst}\;
	\If {$(S\cap \Sg(c,d), \Sg(c,d))$ is a blocker (tested as in Remark~\ref{remTestBlocker})  \label{line:test}}
	  {\Return{$(S\cap \Sg(c,d), \Sg(c,d))$}\label{line:yes}\;}
	  \Else {Let $S:=S\cup \Sg(c,d)$ (note that $S$ need not 
			  be a subalgebra of $\algA$)\;	}
	}
	  \label{line:endwhile}
	}
	\label{line:iterate}
	\Return{No}\;
    \caption{Deciding the existence of a cube term blocker}\label{algBlocker}
  \end{algorithm}
  The idea of Algorithm~\ref{algBlocker} is to look for an, in a sense, ``minimal'' cube term blocker $(C,D)$. If an idempotent algebra $\algA$ contains a cube term blocker $(C,D)$, then we can try to make $C$ and $D$ smaller. Take $c\in C$ and $d\in D\setminus C$; it is easy to verify that the set $(C\cap \Sg(c,d),\Sg(c,d))$ will then also be a cube term blocker. 
  
  It is not too computationally expensive to guess $c,d\in A$ and we can efficiently test if a pair of subsets of $\algA$ is a cube term blocker by looking at basic operations of $\algA$ (see Remark~\ref{remTestBlocker}); hence the only remaining issue is how to construct the right set $C\cap \Sg(c,d)$.
  
  A rough outline of our method for constructing $C\cap\Sg(c,d)$ is as follows: We guess $c\in A$ and then grow a set $S$ as slowly as possible, starting with $S=\{c\}$. In each step we either find a cube term blocker of the form $(S\cap D,D)$, or we grow $S$, so after at most $|A|$ steps we reach $S=A$ and stop.
  
  \begin{theorem}\label{thmIdempotentCTalg}
  Given an idempotent algebra $\algA$ as input, 
  Algorithm~\ref{algBlocker} will in a  polynomial time either find a cube term blocker in $\algA$ or correctly conclude that $\algA$ has a cube term (and thus $\algA$ contains no cube term blocker).
  \end{theorem}
  \begin{proof}
  Examining the pseudo-code of Algorithm~\ref{algBlocker}, it is
  obvious that the algorithm runs in polynomial time -- in the RAM model of
  computation, the time complexity is $O(m|A|^2|\algA|)$ (where $m$ is the maximum arity of a basic operation of $\algA$). It remains to prove
  the correctness of the algorithm.
  
    Since the algorithm tests to see if each potential output is a blocker
    (line~\ref{line:test}), it follows that if Algorithm~\ref{algBlocker} outputs a pair of sets, then this
   pair is a blocker. It remains to show that the algorithm will find a blocker when $\algA$ contains one.

    To this end, assume that $(C,D)$ is a blocker in $\algA$ such that $D$ is
    (inclusion) minimal. Let us fix some $c\in C$. 
    We will now show that when this $c$ is chosen as the value of the variable $c$ in the for-loop
    (lines~\ref{line:for}--\ref{line:iterate}) of 
    Algorithm~\ref{algBlocker} the algorithm will find a blocker. 

    The only way finding a blocker can fail for this particular $c$ is if the inner loop of the algorithm (steps
    \ref{line:while}--\ref{line:endwhile}) eventually adds all
    the elements of $A$ to $S$. The algorithm might find a blocker by ``accident,'' in which case we are done. Assume that this does not happen and let us wait for the first time when on line~\ref{line:choose} a $d\not\in S$ is chosen such that $\Sg(c,d)\cap D\not\subset C$. In other words, we assume that $S\cap D\subset C$ and that $\Sg(c,d)\setminus S$ has a nonempty intersection with $D\setminus C$.
    
    Recall that $D$ is a subuniverse of $\algA$.
    Were $\Sg(c,d)\cap D$ smaller than $\Sg(c,d)$, we could choose an element $d'\in (\Sg(c,d)\cap D)\setminus C$ instead of $d$ and get $d'\not\in S$ with a smaller $\Sg(c,d')\subset D$. Hence from the minimality of $\Sg(c,d)$ we see that $\Sg(c,d)\subset D$ and so $d$ lies in $D$. Since $\Sg(c,d)\not\subseteq C$, the element $d$ must lie in $D\setminus C$. Finally, the minimality of $D$ then gives us $\Sg(c,d)=D$.

    We now claim that $S\cap D=C$. If we prove this, we will be done: The pair
    of sets $( S\cap \Sg( c,d), \Sg( c,d))=(S\cap D, D)=(C,D)$ is a blocker,
    meaning that, instead of adding $d$ to $S$, Algorithm~\ref{algBlocker} will reach
    line~\ref{line:yes} and output $(C,D)$.

We prove $S\cap D=C$ by showing two inclusions. We have $S\cap D\subset C$ by assumption. To see
    $S\cap D\supseteq C$, consider what would happen if there was $a\in C\setminus (S\cap
    D)=C\setminus S$. Then $a\not\in S$
    and yet $\Sg(c,a)\subset C\subsetneq D=\Sg(c,d)$. Therefore, in
    step~\ref{stepInnerFirst}, the set $D=\Sg(c,d)$
    was not minimal and $a$ should have been chosen instead of $d$. This contradiction concludes our proof.


  \end{proof}
  \begin{remark}
    As a
  side note, it turns out that modifying our algorithm to only look for ``nice''
  blockers is tricky: For example it is an $\compNP$-complete problem to decide if,
  given an idempotent algebra $\algA$ and an element $b\in A$, there exists 
  a blocker $(C,D)$ such that $b\in D\setminus C$~\cite{barto-kazda-absorption}.
  \end{remark}
  \section{Cube terms in general algebras}
  Let now $\algA=(A;f_1,\dots,f_\ell)$ be an algebra that is not idempotent. Assuming that $\algA$
  has a cube term, what is the smallest dimension of a cube term that
  $\algA$ has? It is easy to see that
  $\algA$ has a cube term of dimension $n$ if and only if the idempotent reduct of $\algA$
  has an cube term of dimension $n$. We also know that the minimal dimension of a
  cube term is the same as the minimal arity of a near unanimity term -- if
  $\algA$ has a near unanimity, that is.
  
  It turns out that we can recover a bit more from the idempotent case:

  \begin{lemma}\label{lemCCbeNonIdempotent} Let $n\geq 2$ be an integer and
    $\algA$ an algebra containing 
    a cube term of dimension $n+1$, but no cube term of
    dimension $n$ (where $n$ is a positive integer). Then there exists an $n$-ary elusive relation compatible
    with $\algA$.
  \end{lemma}
  \begin{proof}
    Since $\algA$ does not have a cube term of dimension $n$, neither does the
    idempotent reduct $\algA_{idmp}$ of $\algA$. Therefore, there exist tuples
    $\vect a,\vect b\in A^n$ such that $\vect a$ does not lie in the subalgebra
    of $\algA_{idmp}$ generated by $\{\chi_I(\vect a,\vect b)\colon
    I\neq\emptyset\}$. 

    Translating the last sentence from $\algA_{idmp}$ back to $\algA$, we obtain that
    \[
      \seq{A}\vect a\not\in \Sg_{\algA^{|A|+n}}(\{\seq{A}\chi_I(\vect a,\vect b)\colon
      I\neq\emptyset\}).
    \]
  Let now $\vect q$ be the shortest
    tuple of elements of $A$ for which 
 \[
      \vect q\vect a\not\in \Sg_{\algA^{|\vect q|+n}}(\{\vect q\chi_I(\vect a,\vect b)\colon
    I\neq\emptyset\}).
    \]
    By the above reasoning, we have $|\vect q|\leq |A|$. We will show that
    in fact $|\vect q|=0$, proving the Lemma. Assume that $\vect q$ is of length at least
    one; we will show how this leads to a contradiction. 
    
    Let $\vect q=\vect r s$ for a suitable tuple $\vect r$
    and $s\in A$. Denote by $E$ the relation $\Sg(\{\vect q\chi_I(\vect
    a,\vect b)\colon
    I\neq\emptyset\})$. Since we took $\vect q=\vect r s$ shortest possible, we must have
    \[
\vect r\vect a\in \Sg(\{\vect r\chi_I(\vect a,\vect b)\colon
    I\neq\emptyset\}),\quad
\vect rs\vect a\not\in \Sg(\{\vect rs\chi_I(\vect a,\vect b)\colon
    I\neq\emptyset\})=E.
    \]
  We conclude that there exists $u\in A$ such that $\vect r u \vect a\in E$.

    Now since $\algA$ has a cube term of dimension $n+1$, it also has an edge
    term $t$ of dimension $n+1$ (by Theorem~\ref{thmEdgeTerm}). We apply
    $t$ to the following matrix of tuples:
    \[
    \left(\begin{matrix}
    \vect r&\vect r&\vect r&\dots&\vect r&\vect r\\
    s&s&s&\dots&s&u\\
      \chi_{\{1,2\}}(\vect a,\vect b)&\chi_{\{1\}}(\vect a,\vect b)&
	    \chi_{\{2\}}(\vect a,\vect b)&\dots&\chi_{\{n\}}(\vect a,\vect b)
	&\vect a\\
    \end{matrix}\right)
    \]
  This matrix has $n+2$ columns, all of which are in $E$ (the tuples in the first $n+1$ columns 
    are among the tuples that witness elusiveness of $E$ 
    while the last tuple is $\vect r u\vect a\in E$ by the choice of $u$). Thus
    $t$ applied to the matrix outputs a member of $E$. However, using the
    identities for edge terms, one
    can easily verify that the output tuple is in fact $\vect r s \vect a$, a
    contradiction with $\vect rs\vect a\not\in E$.
  \end{proof}

  In the rest of this Section, let $R$ be an $n$-ary relation that is inclusion minimal among all $n$-ary elusive relations compatible with $\algA$. Moreover, we order the coordinates of $R$ so that the two
  tuples that witness the elusiveness of $R$ are of the form
  \begin{align*}
    \vect a&=a_1^{\bf n_1}a_2^{\bf n_2}\dots a_k^{\bf n_k}\\
    \vect b&=b_1^{\bf n_1}b_2^{\bf n_2}\dots b_k^{\bf n_k},
  \end{align*}
   with $(a_i,b_i)\neq (a_j,b_j)$ for $i\neq j$. Call such a pair elusive tuple of \emph{type} $(n_1,n_2,\dots,n_k)$. 
   
   The
   numbers $n_1,\dots,n_k$ give us a partition of $n$ into intervals of
   consecutive integers. As before, we will call the members of this partition
   blocks. The $i$-th block, which we again denote by $B_i$, consists of the indices 
   $B_i=\{j\in \en\colon n_1+\dots+n_{i-1}<j\leq n_1+\dots+n_{i}\}$.
 
%

  \begin{definition}
  Assume that we have a fixed $\algA$, $n_1,n_2,\dots,n_k\in\en$, and $\vect
  a,\vect b\in A^n$ are two tuples of type $(n_1,n_2,\dots,n_k)$.
    Let $D_1,\dots,D_k$ be such that there exists a term $t$ of $\algA$ so that
    $D_i=t(\{a_i,b_i\},\{a_i,b_i\},\dots,\{a_i,b_i\})$ for each $i$, and let
    $C_i\subset D_i$ for each $i$.
    
    Then we define a \emph{blob} of type $(n_1,\dots,n_k)$, given by the sets
    $D_1,\dots,D_k$ and $C_1,\dots,C_k$, as
    \[
      \left\llbracket\begin{matrix}
	C_1&|&D_1^{n_1}\\
	C_2&|&D_2^{n_2}\\
	   &\vdots&\\
	C_k&|&D_k^{n_k}
  \end{matrix}\right\rrbracket
      =
\left\{\vect v\in
	D_1^{n_1}\times D_2^{n_2}\times\dots\times D_k^{n_k} \colon \forall
	i,\,  C_i\subset \{v_j\colon j\in B_i\}\right\}.
	\]
    A union of a family of blobs of type $(n_1,n_2,\dots,n_k)$ is called a \emph{sponge} of type
    $(n_1,n_2,\dots,n_k)$. 
  \end{definition}
  Sponges and blobs are distant relatives of chipped cubes (for example, one
  can write any chipped cube in the form of a sponge). We have shown that
  minimal compatible elusive relations are chipped cubes in the idempotent case; for general algebras, we want
  to show that the minimal compatible elusive relations are sponges (after a suitable reordering of
  coordinates). Before we do that, though, we need to obtain some tools.

  Blobs and sponges carry with them information about blocks of coordinates, so
  we can talk about, say, the second block of coordinates of a blob $\Gamma$.
  We will be talking quite a bit about ``the set of values appearing in a certain
  block of a tuple,'' so let us introduce a short name for this concept:

  \begin{definition}
  For $\vect d\in A^{k}$, we denote by $\cont(\vect d)$ the \emph{content} of
    $\vect d$, which is defined as the set of all elements of $A$ that appear in $\vect d$, i.e.
    $\cont(u_1u_2\dots u_k)=\{u_1,u_2,\dots,u_k\}$. If $\vect v\in A^n$ is a
    tuple and we have a partition of $[n]$ into blocks $B_1,\dots,B_k$ then the content of the
    $i$-th block of $\vect v$ is the set $\{v_j\colon j\in B_i\}$.
  \end{definition}
  In the language of content, a blob is a set of $\vect v\in D_1^{n_1}\times D_2^{n_2}\times\dots\times D_k^{n_k}$ such that for each
  $i=1,\dots,k$, the
  content of the $i$-th block of $\vect v$ contains $C_i$.

 It will also be useful to remap entries of tuples in a prescribed way.
 
 \begin{definition}\label{def:shuffle}
  For a map $\eta\colon [n]\to [n]$ and a tuple $\vect a=(a_1,\dots,a_n)\in A^n$, define the $\eta$-image
  of $\vect a$ as
  $\vect a^\eta=(a_{\eta(1)},a_{\eta(2)},\dots,a_{\eta(n)})$.
  Define the $\eta$-image of $R\subset A^n$, denoted by $R^\eta$,
  as the set of $\eta$-images of all members of $R$. We say that $R$ is
  $\eta$-invariant if $R^\eta\subset R$.
 \end{definition}

  Note that $\eta$ need not be a permutation. 
  It is easy to see that if $\vect{b_1},\dots,\vect{b_k}$ are
  $n$-tuples then $\left(t(\vect{b_1},\dots,\vect{b_k})\right)^\eta
  =t(\vect{b_1}^\eta,\dots,\vect{b_k}^\eta)$. Because $\eta$-images and operations 
  commute, $\eta$-images play nicely with subalgebras:
 
  \begin{observation}\label{obsEtaInvariant}
    Let $\algA$ be an algebra, $n\in \en$, $\eta\colon [n]\to [n]$, $R,S\subset A^n$. Then
    \begin{enumerate}[(a)]
      \item If $R$ is compatible with $\algA$, then $R^\eta$ is compatible with $\algA$,
      \item if $R=\Sg(S)$, then $R$ is $\eta$-invariant if and
  only if $\vect s^\eta\in R$ for each $\vect s\in S$, 
      \item if $S$ is a chipped cube, a sponge or a blob and $\eta$ is a \emph{permutation} that preserves
	the blocks of $S$ (i.e. $\eta$ sends each block $B_i$ of $S$ to itself), then $S^\eta=S$.
      \item if $S$ is a chipped cube, $R=\Sg(S)$, and     $\eta$ is a \emph{permutation} of the coordinates of $S$ that preserves the blocks of 
	$S$ then $R$ is $\eta$-invariant.
    \end{enumerate}
  \end{observation}
  \begin{proof}
    The first point follows in a straightforward way from the definition of a
    compatible relation.
    To prove the second point, realize that the $\eta$-image of $R$ is generated
    by $S^\eta\subset R$.

    To see the third point, consider first the case when $S$ is a chipped cube
    or a blob. The membership of a tuple $\vect d\in A^n$ in $S$ only
    depends on the contents of all blocks of $\vect d$. But the content of the
    $i$-th block of $\vect d$ and $\vect d^\eta$ is the same for all $i$,
    therefore $S^\eta=S$. When $S$ is a sponge, $S$ is just a union of blobs of the
    same type -- we can use the above argument for each blob separately and
    obtain $S^\eta=S$, too.
    
    For the last point, we observe that $R^\eta=\Sg(S^\eta)$ because operations
    commute with $\eta$-images, and that $S^\eta=S$ by the previous point.
    Together, we get $R^\eta=\Sg(S^\eta)=\Sg(S)=R$.
  \end{proof}
  
  The following observation follows directly from the definition of a blob:
    \begin{observation}\label{obsMiniChanges}
      Let $\Gamma$ be a blob and $\vect r\in A^n$. Let $\eta\colon
    [n]\to [n]$ be a mapping that sends each block of $\Gamma$ to itself.
      Assume moreover that the $i$-th block of $\vect r^\eta$
    contains $C_i$ for each $i=1,\dots, k$ (where $k$ is the number of blocks of $\Gamma$). 
    Then $\vect r\in \Gamma$.
  \end{observation}
  \begin{observation}\label{obsEtaChi}
    Let $\vect a,\vect b\in A^n$ be two tuples, $I\subset [n]$ and $\eta\colon
    [n]\to [n]$ a mapping. Then  $\left(\chi_I(\vect a,\vect b)\right)^\eta=\chi_{\eta^{-1}(I)}(\vect
      a^\eta,\vect b^\eta),$
    where $\eta^{-1}(I)=\{j\in [n]\colon \eta(j)\in I\}$.
  \end{observation}
  \begin{proof}
    Take $i\in [n]$. It follows from the definition of $\chi_I$ that the $i$-th entry of both 
$\left(\chi_I(\vect a,\vect b)\right)^\eta$ and $\chi_{\eta^{-1}(I)}(\vect
      a^\eta,\vect b^\eta)$ is equal to $b_{\eta(i)}$ if $\eta(i)\in I$ and
      $a_{\eta(i)}$ otherwise.
  \end{proof}
  It turns out that small elusive relations compatible with the algebra $\algA$ are
  $\eta$-invariant exactly when $\eta$ is a permutation:
  \begin{observation}\label{obsBlocksPerm}
    Let $R=\Sg(\{\chi_I(\vect a,\vect b)\colon I\neq
    \emptyset\})$ be a relation such that
    $\vect a\not\in R$, where
   \begin{align*}
    \vect a&=a_1^{\bf n_1}a_2^{\bf n_2}\dots a_k^{\bf n_k}\\
    \vect b&=b_1^{\bf n_1}b_2^{\bf n_2}\dots b_k^{\bf n_k}.
  \end{align*}
    Let $\eta\colon [n]\to[n]$ preserve the
    blocks of $R$ (given by $n_1,n_2,\dots,n_k$). Then $R^\eta\subset R$ if and
    only if $\eta$ is a permutation.
  \end{observation}
  \begin{proof}
    Note that since $\eta$ preserves the blocks of $R$, we have $\vect
    a^\eta=\vect a$ and $\vect b^\eta=\vect b$.

    If $\eta$ is a permutation, then it sends the set of generators of $R$ to
    itself. Observation~\ref{obsEtaInvariant} then gives us $R^\eta\subset R$. On the
    other hand, assume that $\eta$ is not a permutation. Then $\eta$ is not
    onto; without loss of generality assume that 1 does not lie in the image of
    $\eta$. This gives us a contradiction, though: 
    \[
      \left(\chi_{\{1\}}(\vect a,\vect
    b)\right)^\eta=\chi_{\emptyset}(\vect a^\eta, \vect b^\eta)=\vect
    a^\eta=\vect a\not\in R,
    \]
  where the first equality follows from Observation~\ref{obsEtaChi} and
    $\eta^{-1}(1)=\emptyset$. We see that $\eta$ sends the generator $\chi_{\{1\}}(\vect a,\vect
    b)$ of $R$ outside of $R$, a failure of $\eta$-invariance of $R$.
    
  \end{proof}

  \begin{lemma}\label{lemContainsBlob}
    Let $\algA$ be a finite algebra, and let $R$ be inclusion minimal among all elusive
    relations compatible with $\algA$. Assume that $R=\Sg(\{\chi_I(\vect a,\vect
    b)\colon I\neq \emptyset\})$ and $\vect a\not\in R$ for some tuples $\vect a,\vect b$
    of type $(n_1,n_2,\dots,n_k)$. Let $\vect r=w(M)$ where $w$ is a
    member of the clone of operations of $\algA$ and $M$ is a matrix
    whose columns lie in $\{\chi_I(\vect a,\vect
    b)\colon I\neq \emptyset\}$. Then $R$ contains the blob    
    \[
      \Gamma= \left\llbracket\begin{matrix}
	C_1&|&D_1^{n_1}\\
	C_2&|&D_2^{n_2}\\
	   &\vdots&\\
	C_k&|&D_k^{n_k}
  \end{matrix}\right\rrbracket
    \]
    where $D_i=w(\{a_i,b_i\},\dots,\{a_i,b_i\})$ and $C_i=\{r_j\colon j\in B_i\}$.
  \end{lemma}
  \begin{proof}
    Assume that $\Gamma\setminus R$ is nonempty. We shall show how this yields
    a contradiction. 
    Choose a $\vect d\in \Gamma\setminus R$ and pick an $\eta\colon [n]\to[n]$ that preserves the
    blocks $B_1,\dots,B_k$ and sends $\vect d$ to $\vect r$.  (To construct
    such an $\eta$, let $\eta(j)$ for $j\in B_i$ be any $q\in B_i$ such that
    $d_q=r_j$.) Were $\eta$ a permutation, we would have $\vect
    r^{\eta^{-1}}=\vect d$. But $\eta^{-1}$ would then be a permutation that
    preserves blocks of coordinates, so $\vect d\in R^{\eta^{-1}}\subset R$, a
    contradiction with our choice of $\vect d$. We conclude that
    $\eta$ is not a permutation. By Observation~\ref{obsBlocksPerm}, $R$ is
    not $\eta$-invariant.

    In the rest of the proof, we shall show that $R^\eta\subset R$, which will be a contradiction.

  Let  $\vect c=c_1^{\bf n_1}c_2^{\bf n_2}\dots c_k^{\bf
    n_k}$ where $c_i=w(b_i,\dots,b_i)$. We claim that for any
    nonempty $I\subset [n]$ we have $\chi_I(\vect d,\vect c)\in R$. We know
    that $\vect d=w(N)$ for a suitable matrix $N$ whose first $n_1$ rows contain only
    members of $\{a_1,b_1\}$, the next $n_2$ rows contain only members of
    $\{a_2,b_2\}$, etc. Now rewrite all rows of $N$ whose indices lie in $I$ to
    $b_j$'s -- call the new matrix $N'$. Since $I\neq\emptyset$, each column of
    $N'$ will contain a $b_j$, so columns of $N'$ lie in $R$.  It now remains
    to observe that $w(N')=\chi_I(\vect d,\vect c)\in R$. 
    
    Given that $\vect d\not\in R$, but $\chi_I(\vect d,\vect c)\in R$ for any $I\neq\emptyset$, the
    tuple $\vect d$ is elusive for $R$. What is more, from the minimality of $R$ we get that $R=\Sg(\{\chi_I(\vect d,\vect c)\colon I\neq \emptyset\})$. 
    
    To verify that $R^\eta\subset R$, it is therefore enough to
    show that $\chi_I(\vect d,\vect c)^\eta\in R$ for each $I$. 
    From Observation~\ref{obsEtaChi} and from $\vect d^\eta=\vect r\in R$ and
    $\vect c^\eta=\vect c$, we have 
    \[
      \chi_I(\vect d,\vect c)^\eta=\chi_{\eta^{-1}(I)}(\vect d^\eta,\vect
      c^\eta)=\chi_{\eta^{-1}(I)}(\vect r,\vect c).
      \]
    However, since $\vect r\in D_1^{n_1}\times\dots\times D_k^{n_k}$, we
    can rewrite any set of coordinates of $\vect r$ to coordinates of $\vect c$ and
    stay inside $R$ (this includes rewriting the empty set thanks to $\vect r\in R$). Therefore,
    $\chi_I(\vect  d,\vect c)^\eta=\chi_{\eta^{-1}(I)}(\vect r,\vect c)\in R$,
    making $R$ $\eta$-invariant, a contradiction.
  \end{proof}

  \begin{theorem}\label{thmSponge}
    Let $\algA$ be a finite algebra. Let $R$ be a relation that is inclusion minimal among $n$-ary elusive
    relations compatible with $\algA$. Let the pair of tuples $\vect a,\vect b$
    witness the elusiveness of $R$. 
    Then we can reorder the coordinates of
    $R$ to get a sponge of type $(n_1,\dots,n_k)$ where $k$ is the number of
    distinct pairs $(a_i,b_i)$ for $i=1,\dots,n$.
  \end{theorem}
  \begin{proof}
    Let us group the same pairs of entries of $\vect a, \vect b$ together. That
    is, we permute the coordinates 1, \dots, $n$ so that 
    \begin{align*}
\vect a&=a_1^{\bf n_1}a_2^{\bf n_2}\dots a_k^{\bf n_k}\\
    \vect b&=b_1^{\bf n_1}b_2^{\bf n_2}\dots b_k^{\bf n_k},
  \end{align*}
   where $(a_i,b_i)\neq (a_j,b_j)$ for $i\neq j$. Observe that $k$ is then equal
    to the number of distinct pairs $(a_i,b_i)$, as required.
    Let us denote by $E$ the union of all blobs of
    type $(n_1,\dots,n_k)$ that are contained in $R$. Our goal is to show that
    $E=R$.

    Obviously, $E\subset R$, so all we need to show is that whenever $\vect r\in R$, 
    then $\vect r\in \Gamma$ for some blob $\Gamma\subset R$. Consider any
    $\vect r \in R$. By the minimality of $R$, there is a term operation
    $w$ of $\algA$ such that $\vect r=w(\chi_I(\vect a,\vect b)\colon
    I\neq\emptyset)$. Applying Lemma~\ref{lemContainsBlob}, we get a blob $\Gamma$
    such that $\vect r\in \Gamma\subset R$ -- exactly what we need to show that
    $R=E$.
   \end{proof}

  \begin{lemma}\label{lemUpperBound}
    Let $R$ be a relation inclusion minimal among all elusive relations compatible with $\algA$. Assume moreover that $R$ is also a sponge of
    type $(n_1,n_2,\dots,n_k)$. Let $m$ be the maximum arity of a basic
    operation of $\algA$. If $\algA$ has a cube term, then for each $i$ we
    have $n_i<|A|m$.
  \end{lemma}
  \begin{proof}
    By the minimality of $R$ as an elusive compatible relation, we know that there are two
    tuples $\vect a=a_1^{\bf n_1}a_2^{\bf n_2}\dots a_k^{\bf n_k}$ and $\vect
    b=b_1^{\bf n_1}b_1^{\bf n_2}\dots b_k^{\bf n_k}$ so that $\vect a\not \in R$ and $R=\Sg(\{\chi_I(\vect a,\vect
  b)\colon I\neq \emptyset\})$.
    
    For each $\vect r\in R$, we consider all the blobs of type
    $(n_1,\dots,n_k)$ that the conclusion of
    Lemma~\ref{lemContainsBlob} places inside $R$ (cf. proof of
    Theorem~\ref{thmSponge}). Such blobs cover $R$, so we
    have (for a suitable set
    $L$, and appropriate $C_{i,\ell}$'s and $D_{i,\ell}$'s):
    \[
      R=\bigcup_{\ell\in L}
      \left\llbracket\begin{matrix}
	C_{1,\ell}&|&D_{1,\ell}^{n_1}\\
	C_{2,\ell}&|&D_{2,\ell}^{n_2}\\
	   &\vdots&\\
	C_{k,\ell}&|&D_{k,\ell}^{n_k}
  \end{matrix}\right\rrbracket
    \]

  Assume for a contradiction (and without loss of generality) that $n_1\geq m|A|$. We shall show that
    $\algA$ does not have a cube term. For 
    $s\in\en$ we define the relation $R^{\star s}$ as basically ``$R$ whose
    first block is extended by $s$ entries'':
    \[
      R^{\star s}=\bigcup_{\ell\in L}
\left\llbracket\begin{matrix}
  C_{1,\ell}&|&D_{1,\ell}^{n_1+s}\\
	C_{2,\ell}&|&D_{2,\ell}^{n_2}\\
	   &\vdots&\\
	C_{k,\ell}&|&D_{k,\ell}^{n_k}
  \end{matrix}\right\rrbracket.
    \]
  Let $B^{\star s}_1,B^{\star s}_2,\dots$ be the blocks of $R^{\star s}$; we denote
    the tuples
    $a_1^{\bf n_1+s}a_2^{\bf n_2}\dots$ and $b_1^{\bf n_1+s}b_2^{\bf n_2}\dots$ by $\vect
    a^{\star s}$ and $\vect b^{\star s}$, respectively.

    It is easy to see that $R^{\star s}$ contains $\chi_I(\vect a^{\star s},\vect
    b^{\star s})$ for each nonempty $I\subset [n+s]$, but $\vect a^{\star
    s}\not \in R^{\star s}$, so $R^{\star s}$ is an elusive relation. If we can
    now show that each  $R^{\star s}$ is also compatible with $\algA$, we will have a family of arbitrarily large elusive relations compatible with $\algA$. Therefore, $\algA$ will have no cube term by Lemma~\ref{lemMatrix}.

    Take any basic operation $t$ of $\algA$ of arity $r\leq m$ and let $\vect
    {c_1},\dots, \vect {c_r}\in R^{\star s}$. We want to show that then also $t(\vect
    {c_1},\vect {c_2},\dots,\vect {c_r})\in R^{\star s}$. To simplify notation, we will
    assume that each $\vect {c_i}$ belongs to the $i$-th blob from $L$ (this is
    without loss of generality, as we can reorder blobs and even take several
    copies of the same blob without changing $R^{\star s}$). 
  
    Let us now examine the content of the first block of entries of $\vect
    {c_1},\dots,\vect {c_r}$. Let  $\vect {d_i}=(\vect {c_i})_{[1,n_1+s]}$ for
    $i=1,\dots,r$. Then $\cont(\vect {d_i})$ has size at most $|A|$ for each
    $i=1,\dots,r$ and thus
    there are at most $r|A|$ indices in $B_1^{\star s}$ that witness the content of
    \emph{all} $\vect {d_1},\dots,\vect {d_r}$. Since $n_1\geq r|A|$ and $R^{\star s}$ is invariant under
    permuting the first block, we can assume that the complete content of all $\vect
    {d_i}$'s appears in the last $n_1$ entries of $B_1^{\star s}$. That is, for each
    $i=1,\dots, r$ we have
    \[
     \cont((\vect {c_i})_{[1,s+n_1]})= \cont((\vect {c_i})_{[s+1,s+n_1]}).
    \]
    
  Let now $\vect{e_i}=(\vect {c_i})_{[s+1,n]}$, i.e. $\vect{e_i}$ is obtained from
    $\vect{c_i}$
    by cutting away the first $s$ entries. Since the contents of the first
    blocks remain the same, we see
    \[
      \vect{e_i}  \in
\left\llbracket\begin{matrix}
  C_{1,i}&|&D_{1,i}^{n_1}\\
	C_{2,i}&|&D_{2,i}^{n_2}\\
	   &\vdots&\\
	C_{k,i}&|&D_{k,i}^{n_k}
  \end{matrix}\right\rrbracket,
  \]
so $\vect{e_i}\in R$. By the definition of $R$, for each $i$ there exists a
$(2^n-1)$-ary operation $u_i$ in the
    clone of $\algA$ such that 
    $\vect{e_i}=u_i(\chi_I(\vect a,\vect b)\colon I\neq \emptyset)$.
    
    Therefore, $t(\vect{e_1},\dots,\vect{e_r})\in R$. Denote by $w$ the
    $r(2^n-1)$-ary term $t\circ (u_1,u_2,\dots,u_r)$ (where each $u_i$ has its own distinct variable set).
    Let $M$ be the $n\times r(2^n-1)$ matrix that consists of $r$ copies of 
    $(\chi_I(\vect a,\vect b)\colon I\neq \emptyset)$ arranged next to each
    other. It is easy to verify that $w(M)=t(\vect {e_1},\dots,\vect {e_k})$.
    Using
    Lemma~\ref{lemContainsBlob} with the term $w$ and tuple $w(M)$, we obtain 
    \[
      t(\vect{e_1},\dots,\vect{e_r})\in\left\llbracket\begin{matrix}
	C_1&|&D_1^{n_1}\\
	C_2&|&D_2^{n_2}\\
	   &\vdots&\\
	C_k&|&D_k^{n_k}
  \end{matrix}\right\rrbracket
      \subset R,
    \]
  where $C_i$ is the content of the $i$-th block of $w(M)$ and we let
    $D_i$ to be equal to $w(\{a_i,b_i\},\dots,\{a_i,b_i\})$. Since the tuple $t(\vect
    {e_1},\dots,\vect{e_r})$ is a suffix of $t(\vect
    {c_1},\dots,\vect{c_r})$, we only need to verify that
$t((\vect{c_1})_j,\dots,(\vect{c_r})_j)\in D_1$
    for each $j=1,\dots,s$ to obtain
    \[
      t(\vect{c_1},\dots,\vect{c_r})\in\left\llbracket\begin{matrix}
	C_1&|&D_1^{n_1+s}\\
	C_2&|&D_2^{n_2}\\
	   &\vdots&\\
	C_k&|&D_k^{n_k}
      \end{matrix}\right\rrbracket \subseteq R^{\star s},
    \]
  finishing the proof.

  Pick a $j\in \{1,\dots,s\}$.
  We know that $(\vect {c_1})_j\in \cont(\vect
    e_1)_{[1,n_1]}$ and so 
    \[
      (\vect {c_1})_j\in u_1(\{a_1,b_1\},\dots,\{a_1,b_1\}).
      \]
      We can do the same thing
    with $(\vect {c_2})_j$, $(\vect {c_3})_j$ and so on, getting
    \begin{align*}
      t((\vect{c_1})_j,\dots,(\vect{c_r})_j)&\in t(u_1(\{a_1,b_1\},\dots),\dots,u_r(\{a_1,b_1\},\dots))\\
      &\in w(\{a_1,b_1\},\{a_1,b_1\},\dots)=D_1,
    \end{align*}
  which is exactly what we needed.
  \end{proof}

  \begin{theorem}\label{thmGeneralCTB}
    Let $\algA$ be a finite algebra with the universe $\{1,2,\dots,|A|\}$. 
    Let $m$ be the maximal arity of a basic operation in $\algA$. Then $\algA$
    has a cube term if and only if $\algA$ has a cube term of dimension
    at most $|A|^3m$.
  \end{theorem}
  \begin{proof}
    The nontrivial implication is ``$\Rightarrow$.'' Observe that the theorem 
    is true (but not very interesting) when $|A|=1$ or $m=1$. Thus we let $|A|^3m>2$.

    Assume that $\algA$ has a cube term. Using
    Lemma~\ref{lemCCbeNonIdempotent}, it is enough to show that every
    elusive relation compatible with $\algA$ has arity less than $|A|^3m$.

    Let $R$ be an $n$-ary elusive relation compatible with $\algA$. 
    By taking $R$ inclusion minimal and applying Theorem~\ref{thmSponge}, we can assume without loss of generality that $R$ is a sponge of type $(n_1,\dots,n_k)$ where
    $n_1+n_2+\dots+n_k=n$ and $k\leq |A|^2$.

    By Lemma~\ref{lemUpperBound} we have $n_i<|A|m$. Therefore
    \[
      n=\sum_{i=1}^k n_i<k|A|m\leq
    |A|^2|A|m=|A|^3m,
    \]
  giving us $n<|A|^3m$ as was needed.
  \end{proof}

  \section{Deciding cube and near unanimity terms is in \compEXPTIME{}}\label{secGeneralAlg}
 We conclude our paper with two algorithmic corollaries of Theorem~\ref{thmGeneralCTB}.
 \begin{corollary}\label{cor:edge-in-exptime}
  The problem of deciding if a given finite algebra (given by its tables of basic operations) has a cube term is in \compEXPTIME.
 \end{corollary}
 \begin{proof}
  We present an \compEXPTIME{} algorithm for the problem. Let $d=|A|^3m$, where $m$ is the maximum arity of a basic operation of $\algA$.
  For each pair $a, b\in A$, consider the subuniverse $R_{a,b}$ of $\algA^{|A|+m|A|^3}$ generated by 
  $\seq{A}\chi_I(a^{\vect d},b^{\vect d})$ for all $I\neq \emptyset$.

  If $R_{a,b}$ does not contain the tuple $\seq{A}a^{\vect d}$, then
  $\algA$ has no $d$-dimensional cube term, and so no cube term at all by
  Theorem~\ref{thmGeneralCTB}. On the other hand, if $\algA$ has no cube term, then let $\algA_{idmp}$ be the idempotent reduct of $\algA$. We see by Proposition~\ref{propMMMblocker} that $\algA_{idmp}$ has a cube term blocker $(C,D)$. Pick $b\in C$ and $a\in D\setminus C$; from the definition of a blocker it follows that the subuniverse of $\algA_{idmp}^d$ generated by $\chi_I(a^{\vect d},b^{\vect d})$ for $I\neq \emptyset$ does not contain $a^{\vect d}$. Translating this back to $\algA$, we obtain that $R_{a,b}$ does not contain $a^{\vect d}$, finishing the proof of correctness of our algorithm.
  
If $\algA$ has $k$ basic operations of
arity at most $m$, then $|\algA|=O(k|A|^m)$ and
generating each $R_{a,b}$
takes time $O\left(mk\left(|A|^{|A|+m|A|^3}\right)^m\right)$. Assume for
simplicity that $k$ is much less
than $2^{|A|}$. The time estimate for generating $R_{a,b}$ then becomes $2^{O(m^2|A|^3\log{|A|})}$. Since there are only $|A|^2=2^{O(\log{|A|})}$ many possible choices of $a,b$, the total running time of our algorithm will also be
$2^{O(m^2|A|^3\log|A|)}$, placing the problem into the \compEXPTIME{} complexity class.
\end{proof}
The algorithm above can be made a bit faster, but only at the cost of making it
more complicated (by suitably generalizing the notion of cube term blocker). Even the best algorithm known to us still needs exponential time in the worst
case. Since building up the machinery for generalized cube term blockers
would take up several more pages and the complexity theoretic payoff is small, we have decided to present only the
straightforward algorithm here.

As noted in the Introduction, an algebra has a near  unanimity term if and only if it has a cube term and generates a congruence distributive variety. Deciding if a given algebra generates a congruence distributive variety lies in \compEXPTIME{} (in fact, it is \compEXPTIME{}-complete~\cite{freese-valeriote-maltsev-conditions}). Therefore, it follows from Corollary~\ref{cor:edge-in-exptime} that deciding if an algebra has a near unanimity term is in \compEXPTIME.

\begin{corollary}
The problem of deciding if a given finite algebra (given by tables of its basic operations) has a near unanimity term is in \compEXPTIME.
 \end{corollary}

\section*{Conclusions}
We have shown that there are strong conditions that limit the way in which a
finite algebra with finitely many basic operations
can have a cube term. However, there are still open problems associated
with this topic.

First of all, given a signature and a finite algebra $\algA$ of this signature with a cube term, what is the minimal dimension of a cube term in the clone of $\algA$? The construction for the
idempotent case gives us a tight lower bound on the minimal dimension of cube terms.
It is possible to do some minor improvements upon the $m|A|^3$ upper
bound we have presented, but it remains open if there are large algebras with
minimal cube term dimension $\Omega(m|A|^3)$.

It is often the case that various Maltsev conditions that are polynomial time
decidable for idempotent algebras turn out to be hard for general algebras,
see~\cite{freese-valeriote-maltsev-conditions}. We conjecture that one can not
do much better than the algorithm presented in Section~\ref{secGeneralAlg}, i.e.
that the problem of deciding whether a given algebra has a cube term is, like the decision problems for many other Maltsev conditions, \compEXPTIME-complete.

\bibliographystyle{spmpsci}
\bibliography{citations}
\end{document}